\documentclass[12pt,reqno]{amsart}
\usepackage[T2A]{fontenc}
\usepackage[cp1251]{inputenc}
\usepackage[english]{babel}

\usepackage[pdftex,unicode,colorlinks,linkcolor=blue,
citecolor=red,bookmarksopen,pdfhighlight=/N]{hyperref}

\usepackage{indentfirst}
\usepackage{verbatim}
\usepackage{amsfonts,amssymb,mathrsfs,amscd}
\usepackage{amsmath,latexsym}
\usepackage{amsthm}
\usepackage{verbatim}
\usepackage{graphicx}
\usepackage{enumerate}
\usepackage{textcomp}

\theoremstyle{plain} 
\newtheorem{theorem}{Theorem} 
\newtheorem{lemma}{Lemma}[section]
\newtheorem{proposition}{Proposition}[section] 
\newtheorem{corollary}{Corollary}[section] 

\newtheorem{dualtheorem}{Duality Theorem}
\newtheorem{dualtheoA}{Duality Theorem A} 
\newtheorem{dualtheoB}{Duality Theorem B} 
 
\newtheorem{gluingtheoremA}{Gluing Theorem A}
\newtheorem{theoC}{Theorem C} 
\newtheorem{gluingtheorem}{Gluing Theorem} 
\newtheorem{Criterium}{Criterium} 
\newtheorem{PLf}{\bf Poincar\'e\,--\,Lelong formula}

\theoremstyle{definition}
\newtheorem{definition}{Definition}[section] 
\newtheorem{remark}{Remark}[section] 
\newtheorem{example}{Example}[section]

\newcommand{\RR}{\mathbb{R}} 
\newcommand{\CC}{\mathbb{C}} 
\newcommand{\NN}{\mathbb{N}} 
\newcommand{\BB}{\mathbb{B}}  
\newcommand{\const}{{\rm const}} 
\newcommand{\pt}{{\rm pt}} 
\newcommand{\dd}{\,{\rm d}}
\DeclareMathOperator{\inhull}{\rm hull-in}
\DeclareMathOperator{\diam}{diam}
\DeclareMathOperator{\dist}{dist} 
\DeclareMathOperator{\clos}{clos} 
\DeclareMathOperator{\Int}{int}
\DeclareMathOperator{\Meas}{Meas}

\DeclareMathOperator{\har}{har} 
\DeclareMathOperator{\Hol}{Hol}
 
\DeclareMathOperator{\comp}{cmp}
\DeclareMathOperator{\Zero}{Zero} 
\DeclareMathOperator{\sbh}{sbh}

\DeclareMathOperator{\dsbh}{\text{$\delta${\rm -sbh}}} 
\DeclareMathOperator{\supp}{supp} 
\DeclareMathOperator{\loc}{loc}
\DeclareMathOperator{\sgn}{sgn} 
\DeclareMathOperator{\Borel}{Borel}
\DeclareMathOperator{\dom}{dom}
\DeclareMathOperator{\Dom}{Dom}
\DeclareMathOperator{\conn}{conn}
\DeclareMathOperator{\Conn}{Conn}

\title[Affine Balayage of Measures]{Affine Balayage of Measures and Distribution\\  of Riesz Measures of Subharmonic Functions}
\author{Bulat Khabibullin  \and Enzhe Menshikova}
\email
{khabib-bulat@mail.ru, algeom@bsu.bashedu.ru}
\thanks{The work was supported by a Grant of the Russian Science Foundation (Project No. 18-11-00002), and by a Grant   of the Russian Foundation of Basic Research (Project No. 19-31-90007).}

\begin{document}

\maketitle

\tableofcontents

\section{Introduction}\label{int}

We have considered in the survey \cite{KhaRozKha19} general concepts of \textit{affine  balayage.\/} In this article we deal with a particular case of such balayage with respect to special  classes of  test subharmonic  functions. This allows us to generalize and develop  some results from \cite{KhaRoz18},~\cite{MenKha19},~\cite{KhaKha19}.

The general concept of balayage can be defined as follows. Let $F$ be a set and $(R,\leq)$ be a   (pre-)ordered set with (pre-)order relation $\leq$. A function $f\colon F\to R$
can be called a (linear) {\it  balayage} of a function $g\in F\to R$ {\it  for a subset\/} $\mathcal V \subset F$, and we write $g\preceq_{\mathcal V}  f$,  if the function $f$ majorizes the function $g$ on $\mathcal V $:
\begin{equation}\label{b0}
g(v)\leq f(v) \quad \text{for all $v\in \mathcal V$} .
\end{equation}
Suppose, in addition, that $R$ is the {\it extended real line}  with usual order relation~$\leq$.  A function $f\colon F\to R$ can be called an {\it affine  balayage} of a function $g\in F\to R$ {\it for a subset\/} $\mathcal V \subset F$, and we write $g\curlyeqprec_{\mathcal V} f$,  if there is a real constant $C$ such that the function $f+C$ majorizes the function $g$ on $\mathcal V$:
\begin{equation}\label{ba}
g(v)\leq f(v)+C \quad \text{for all $v\in \mathcal V$} .
\end{equation}
In this article, we use the balayage or the affine balayage when $F$ is a class of functions on a subdomain $D$ of  finite-dimensional Euclidean space, functions $f$ and $g$ are integrals  defined by positive measures on $D$, and classes  $\mathcal V$ are  special classes of subharmonic functions on $D$ or  near the boundary of this domain $D$  from the inside, respectively (Sec. \ref{Ssec_balm}--Sec. \ref{afm}).

 Using this special cases of balayage and affine balayage, we investigate two related but different problems.  Let $u\not\equiv -\infty$ and $M\not\equiv -\infty$ be a pair  subharmonic functions on a domain $D$ in  $d$-dimensional Euclidean space.
The first is to find the relations between the Riesz measures $\upsilon_u$ and $\mu_M$ of functions $u$ and $M$ respectively under which there exists a subharmonic function $h\not\equiv -\infty$ on $D$ such that $u+h<M$. 
The second is the same question, but for a harmonic function $h$ on $D$. 
The answers to these questions are given in terms of affine balayage.  
Such  function $h$ exists if and only if the measure $\mu_M$ is an affine balayage of the measure $\upsilon_u$ for a special class $\mathcal V$ of subharmonic test functions defined on $D\setminus S_o$, where $S_o$ is  some precompact fixed subset in $D$ (Sec. \ref{Crit}, Criteria \ref{crit1} and \ref{crit2}).

Applications of these results relate to conditions on the distribution of zeros of  a holomorphic functions $f$ on a subdomain $D$  of $n$-dimensional complex space under the restriction $|f|\leq \exp M$, where $M$ is a $\delta$-subharmonic 
function on $D$ (Sec. \ref{secH}, Theorem \ref{ThHol}). For finitely connected domains $D$ in the complex plane (Subsec. \ref{Dfd}), these descriptions are complete (Criterium \ref{crit3}) or almost complete (Theorem \ref{th:3} together with Theorem \ref{ThHol}).

Obtained and used  auxiliary results may have independent significance. It is primarily the Gluing Theorems \ref{gl:th2} and \ref{gl:th3} for subharmonic functions (Sec. \ref{GT}) with  Green's functions (Sec. \ref{GTg}, Gluing Theorems \ref{gl:th4} and \ref{gl:th_es}), properties of linear  balayage of measures and charges (Sec. \ref{Ssec_balm}, Propositions \ref{Prtr}--\ref{Pr_pol}, Examples \ref{sbhJ}--\ref{5}), an internal description for potentials  of balayage of measures (Sec. \ref{Pot}), including Duality Theorems \ref{DT1} and \ref{DTsbh}  for them together with Duality Theorems A and B from Subsections \ref{ASP} and \ref{JP} for Arens\,--\,Singer and Jensen measures and potentials, as well as our  generalized Poisson\,--\,Jensen formula   (Subsec. \ref{PJ}, Theorem \ref{PJf}). 
Because of this the auxiliaries results are often proved in a more general form than is necessary for the main purposes of this article.

\section{Definitions, notations and conventions}\label{Ss12}
\setcounter{equation}{0}

The reader may address to this Subsec. \ref{Ss12} when necessary.

\subsection{Sets, order, topology}\label{SsSets} 

As usual, $\mathbb N:=\{1,2, \dots\}$,  $\mathbb R$ and $\mathbb C$ are the sets 
of all {\it natural, real \/} and {\it complex\/} numbers, respectively; 
$\NN_0:=\{0\}\cup\NN$   is French natural series.

For $d \in \NN$, we  denote by $\mathbb R^d$ the {\it $d$-dimensional real  Euclidean  space\/} with the standard {\it Euclidean norm\/} $|x|:=\sqrt{x_1^2+\dots+x_d^2}$ for $x=(x_1,\dots ,x_d)\in \RR^d$
and the distance function $\dist (\cdot, \cdot)$.
For  the {\it real line\/} $\RR=\RR^1$ with  {\it Euclidean norm-module\/} $|\cdot |$,   
\begin{subequations}\label{df:R}
\begin{align}
\RR_{-\infty}:=\{-\infty\}\cup \RR,\; 	\RR_{+\infty}:=\RR\cup 
\{+\infty\}, \; &|\pm\infty|:=+\infty; 
\;  \RR_{\pm\infty}:=\RR_{-\infty}\cup \RR_{+\infty}
\tag{\ref{df:R}$_\infty$}\label{df:Rr}\\
\intertext{is {\it extended real line\/} in the end topology with two ends $\pm \infty$, with   the order relation $\leq$ on $\RR$ complemented by the 
inequalities $-\infty \leq x\leq +\infty$ for $x\in \RR_{\pm\infty}$, with the {\it positive real axis}}
\RR^+:= \{x\in \RR\colon x\geq 0\}, \; \RR_{+\infty}^+:=\RR^+\cup\{+\infty\}, &\; \begin{cases}
x^+&:=\max\{0,x \},\\
 x^-&:=(-x)^+,
\end{cases}
\; \text{for $x\in \RR_{\pm\infty}$},
\tag{\ref{df:R}$^+$}\label{df:R+}
\\
S^+:=\{x\geq 0\colon x\in S \}, \quad S_*:=S\setminus \{0 \} &\quad\text{for $S\subset \RR_{\pm\infty}$}, \quad \RR_*^+:=(\RR^+)_*, 
\tag{\ref{df:R}$_*^+$}\label{df:R*}\\
x\cdot (\pm\infty):=\pm\infty=:(-x)\cdot (\mp\infty)& \quad \text{for $x\in \RR_*^+\cup (+\infty)$}, 
\tag{\ref{df:R}$_\pm$}\label{{infty}+}\\
\frac{x}{\pm\infty}:=0\quad\text{for $x\in  \RR$},&\quad  \text{but $0\cdot (\pm\infty):=0$}
\tag{\ref{df:R}$_0$}\label{{infty}0}
\end{align}
\end{subequations} 
unless otherwise specified. An open connected (sub-)set of $\RR_{\pm\infty}$  is a {\it  (sub-)interval}  of $\RR_{\pm\infty}$.
The {\it  Alexandroff\/} one-point {\it compactification\/} of $\mathbb R^d$ is denoted by $\mathbb R^d_{\infty}:=\mathbb R^d \cup \{\infty\}$.

The same symbol $0$ is used, depending on the context, to denote the number zero, the origin, zero vector, zero function, zero measure, etc. The {\it positiveness\/} is everywhere understood as $\geq  0$ according to the context.
Given $x\in \RR^d$ and\footnote{A reference mark over a symbol of (in)equality, inclusion, or more general binary relation, etc. means that this relation is somehow related to this reference.}  $r\overset{\eqref{df:R+}}{\in} \RR_{+\infty}^+$, we set 
\begin{subequations}\label{B}
\begin{align}
B(x,r):=\{x'\in \RR^d \colon |x'-x|<r\},&
\quad \overline{B}(x,r):=\{x'\in \RR^d \colon |x'-x|\leq r\},
\tag{\ref{B}B}\label{{B}B}
\\
\quad B(\infty,r):=\{x\in \RR_{\infty}^d \colon |x|>1/r\},&\quad 
\overline B(\infty,r):=\{x\in \RR_{\infty}^d \colon |x|\geq 1/r\},
\tag{\ref{B}$_\infty$}\label{{B}infty}
\\
B(r):=B(0,r),\quad \BB:=B(0,1),& \quad \overline{B}(r):=\overline{B}(0,r),
\quad \overline \BB:=\overline B(0,1).
\tag{\ref{B}$_1$}\label{{B}1}
\\
B_{\circ}(x,r):=B(x,r)\setminus \{x\} ,&\quad 
 \overline{B}_{\circ}(x,r):=\overline{B}(x,r)\setminus \{x\}.
\tag{\ref{B}$_\circ$}\label{Bo}
\end{align}
\end{subequations} 
Thus, a basis of open (respectively closed) neighborhoods of the point $x \in \RR_{\infty}^d$ consists of all  {\it open\/} (respectively {\it closed\/}) {\it balls\/} 
$B(x,r)$ (respectively $\overline B(x,r)$) centered at  $x$ with radius $r>0$.

Given a subset $S$ of $\RR^d_{\infty}$, the \textit{closure\/} 
$\clos S$, the\textit{ interior\/} $\Int S$  and the \textit{boundary\/} $\partial S$ will always be taken relative $\RR^d_{\infty}$. For $S'\subset S\subset \RR^d_{\infty}$ we write  
$S'\Subset S$ if $\clos S'\subset \Int S$.   
An open connected (sub-)set of $\RR^d_{\infty}$  is a {\it  (sub-)domain}  of $\RR^d_{\infty}$. 

\subsection{Functions.}\label{Functions} Let $X,Y$ are sets. We denote by $Y^X$ the set of all functions  $f\colon X\to Y$. The value $f(x) \in Y$ of an arbitrary function $f\in Y^X$ is not necessarily defined for all $x \in X$. The restriction of a function f to $S \subset X$ is denoted by $f\bigm|_{S}$. If $F\subset Y^X$, then $F\bigm|_S:=
\{ f\bigm|_{S} \colon f\in F\}$. We set 
\begin{equation}\label{RX}
\RR_{-\infty}^X\overset{\eqref{df:Rr}}{:=}(\RR_{-\infty})^X, \quad
\RR_{+\infty}^X\overset{\eqref{df:Rr}}{:=}(\RR_{+\infty})^X,\quad
\RR_{\pm\infty}^X\overset{\eqref{df:Rr}}{:=}(\RR_{\pm\infty})^X.  
\end{equation}
A function $f\in \RR_{\pm\infty}^X$ is said to be {\it extended numerical.\/}  
For extended numerical functions $f$, we set 
\begin{equation}\label{dom}
\begin{split}
\Dom_{-\infty}:=f^{-1}(\RR_{-\infty})\subset X,& 
\quad  \Dom_{+\infty} f:=f^{-1}(\RR_{+\infty})\subset X, \\
  \Dom f:=f^{-1}(\RR_{\pm\infty})&=
\Dom_{-\infty} f\bigcup \Dom_{+\infty}f\subset X,\\
\dom f:=f^{-1}(\RR)=&\Dom_{-\infty} f\bigcap \Dom_{+\infty}f\subset X, 
\end{split}
\end{equation} 
For $f,g\in \RR_{\pm\infty}^X$  we write $f= g$ if  
$\Dom f=\Dom g=:D$ and $f(x)=g(x)$ for all $x\in D$, 
and we write $f\leq g$ if $f(x)\leq g(x)$ for all $x\in D$.
For $f\in \RR_{\pm\infty}^X$, $g\in \RR_{\pm\infty}^X$  and a set $S$, we write 
``$f = g$ {\it on\/} $S$\,'' or  ``$f \leq g$ {\it on\/} $S$\,'' if 
 $f\bigm|_{S\cap D}= g\bigm|_{S\cap D}$ or $f\bigm|_{S\cap D}\leq g\bigm|_{S\cap D}$ respectively.

For $f\in F\subset \RR_{\pm\infty}^X $, we set $f^+\colon x\mapsto \max \{0,f(x)\}$,
$x\in \Dom f$,  $F^+:=\{f\geq 0 \colon f\in F\}$. So, $f$  is \textit{positive\/} on $X$ if $f=f^+$, and  we write ``$f\geq 0$ {\it on\/} $X$''.We will use the following construction of {\it countable completion of $F$ up:}
\begin{multline}\label{Fuparrow}
F^{\uparrow}:=\bigl\{ f\in \RR_{\pm\infty}^X \colon \text{\it there is an increasing 
 sequence $(f_j)_{j\in \NN}$, $f_j\in F$,}\\
\text{\it such that $f(x)=\lim_{j\to \infty} f_j(x)$ for all $x\in X$\/ {\rm (we write $f_j\underset{j\to \infty}{\nearrow} f)$}} \bigr\}.
\end{multline}
\begin{proposition}\label{pr:up}
Let $F\subset \RR_{\pm}^X$ be a subset closed relative to the maximum.  Consider sequences
$F\ni f_{kj}\underset{j\to \infty}{\nearrow} f_k\underset{k\to \infty}{\nearrow} f$.
Then $F\ni  \max\{ f_{kj}\colon  j\leq n, k\leq n\}\underset{n\to \infty}{\nearrow}f$.
In particular, $(F^\uparrow)^{\uparrow}=F^\uparrow$.
\end{proposition}
The proof is obvious.

For topological space $X$, $C(X)\subset \RR^X$  is the vector space over $\RR$ 
of all continuous functions. We denote the function identically equal to resp. $-\infty$ or $+\infty$ on a set  by the same symbols $-\infty$ or $+\infty$.

For an open set  $O\subset \RR^d_{\infty}$,
 we denote  by $\har (O)$ and  $\sbh (O)$  the classes  of all {\it harmonic\/} (locally affine for $d = 1$) and    {\it subharmonic\/} (locally convex for $d = 1$) functions on $O$, respectively.  The class $\sbh ( O)$  contains the {\it minus-infinity function\/} 
$-\infty$; 
\begin{equation}\label{sbh}
 \sbh_*(  O):=\sbh\,(  O)\setminus 
\{\boldsymbol{-\infty}\}, \quad 
	\sbh^+(  O):=( \sbh (  O))^+.
\end{equation}
Denote by $\dsbh (O):=\sbh(O)-\sbh (O)$ the class of all {\it $\delta$-subharmonic\/} functions on $O$ \cite{Arsove}, \cite[3.1]{KhaRoz18}. The class $\dsbh (  O)$  contains two trivial functions, $\boldsymbol{-\infty}$ and 
$\boldsymbol{+\infty}:=-(\boldsymbol{-\infty})$;
\begin{equation}\label{dsbh} 
\dsbh_*(O)\overset{\eqref{sbh}}{:=}\dsbh(O)\setminus  \{ \boldsymbol{\pm\infty}\}.
\end{equation} 
 If $o\notin O\ni \infty$, then we can to use the  \textit{inversion\/} in  the 
sphere $\partial B(o,1)$ centered at $o \in  \RR^d$:
\begin{subequations}\label{stK}
\begin{align}
\star_{o} \colon x\longmapsto x^{\star_{o}}&:= \begin{cases}
o\quad&\text{for $x=\infty$},\\
o+\frac{1}{|x-o|^2}\,(x-o)\quad&\text{for $x\neq o,\infty$},\\
\infty\quad&\text{for $x=o$},
\end{cases}
\qquad \star:=\star_0=:\star_{\infty}
\tag{\ref{stK}$\star$}\label{stK*}
\\
\intertext{together with the  {\it Kelvin transform\/} \cite[Ch. 2, 6; Ch. 9]{Helms}
}
u^{\star_o}(x^{\star_o})&=|x-o|^{d-2}u(x), \quad x ^{\star_o}\in   
O^{\star_o}:=\{x^{\star_o}\colon x\in
  O\}, 
\tag{\ref{stK}u}\label{stKu}
\\
&\Bigl(u\in \sbh (O)\Bigr)\Longleftrightarrow  \Bigl(u^{\star_o}\in \sbh (O^{\star_o})\Bigr).
\tag{\ref{stK}s}\label{stKs}
\end{align}
\end{subequations}

For a  subset $S\subset \RR_{\infty}^d$,  the classes $\har (S)$, $\sbh(S)$, $\dsbh (S):=\sbh(S)-\sbh (S)$, and $C^k(S)$ with $k\in \NN\cup \{\infty\}$ consist of the restrictions to $S$ of {\it harmonic, subharmonic, $\delta$-subharmonic,{\rm  and}
  k times continuously differentiable functions\/} in some (in general, its own for each function) open set $O\subset \RR_{\infty}^d$ containing $S$.   Classes 
$\sbh_*(S)$, $\dsbh_* (S)$ are defined like previous classes \eqref{sbh}, \eqref{dsbh},
\begin{equation}\label{+S}
\sbh^+(S)\overset{\eqref{sbh}}{:=}\bigl\{u\in \sbh(S)\colon u\geq 0 \text{ on } S\bigr\}.
\end{equation}

By $\const_{a_1,a_2,\dots}\in \RR$ and $\const_{a_1,a_2,\dots}^+\in \RR^+$ we denote constants, and constant functions, in general, depend on $a_1,a_2,\dots$ and, unless otherwise specified, only on them,
where the dependence on dimension $d$ of $\RR_{\infty}^d$ will be not specified and not discussed.

\subsection{Measures and charges.} 

Let $\Borel (S)$ be the class of all Borel subsets in $S\subset \in \RR_{\infty}^d$, and $S\in \Borel(\RR_{\infty}^d)$. We denote by $\Meas(S)$  the class of all Borel signed measures, or, \textit{charges\/} on $S\in{\Borel} (\RR_{\infty}^d)$;   $\Meas_{\comp}(S)$ is the class of charges $\mu \in \Meas(S)$ with a compact support $\supp \mu \Subset S$; 
\begin{subequations}\label{m}
\begin{align}
\Meas^+(S)&:=\{ \mu\in  \Meas (S)\colon \mu\geq 0\},
 \; \Meas_{\comp}^+(S):= \Meas_{\comp} (S)\cap \Meas^+(S);
\tag{\ref{m}$^+$}\label{m+}\\ 
  \Meas^{1+}(S)&:=\{\mu \in \Meas^+(S)\colon \mu(S)=1 \} \text{, \;  \it  probability measures}.
\tag{\ref{m}$^1$}\label{m1}
\end{align}
\end{subequations}
For a charge $\mu \in \Meas(S)$, $\mu^+$, $\mu^-:=(-\mu)^+$ and $|\mu| := \mu^+ +\mu^-$ are  its {\it upper, lower,\/} and {\it total variations,} respectively.  So, $\delta_x \in \Meas_{\comp}^{1+} (S)$
is the {\it Dirac measure\/} at a point $x \in S$, i.e., $\supp \delta_x = \{x\}$, $\delta_x (\{x\}) = 1$. We denote by  $\mu\bigm|_{S'}$
the restriction of $\mu$ to  $S'\in {\Borel} (\RR_{\infty}^d)$.

If the Kelvin transform \eqref{stK} translates the subharmonic function $u$ into another function $u^{\star}_o$ \eqref{stKu}, then its Riesz measure $\upsilon$ is transformed  common use  image under its own mapping-inversion of type $1$ or $2$. These rules are described in detail in L. Schwartz's  monograph \cite[Vol.~I, Ch.IV, \S~6]{Schwartz} and we do not dwell on them here.

Given $S\in{\Borel}(\RR_{\infty}^d)$ and $\mu\in \Meas (S)$,  the class $L^1_{\loc} (S, \mu)$ consists of all extended numerical locally integrable functions with respect to the measure $\mu$ on $S$. For the Lebesgue measure $\lambda_d$, we set  $L^1_{\loc} (S):=L^1_{\loc} (S, \lambda_d)$. For $ L\subset L^1_{\loc}(S,\mu )$, we define a subclass
\begin{equation}\label{MLl}
L \dd \mu:=\bigl\{\nu \in  \Meas (S)\colon 
\text{\it  there exists $g\in  L$ such that\/   $\dd \nu=g \dd \mu$} \bigr\}
\end{equation} 
 of the class of all absolutely continuous charges with respect to $\mu$. 
For $\mu \in \Meas(S)$, we set  
\begin{equation}\label{mB}
\mu(x,r):=\mu\bigl(B(x,r)\bigr) \text{ if $B(x,r)\overset{\eqref{B}}{\subset} S$}.
\end{equation}
Let ${\bigtriangleup}$  be the  the {\it Laplace operator\/}  acting in the sense of the
theory of distributions, $\Gamma$ be the \textit{gamma function,}
\begin{equation}\label{sd-1}
s_{d-1}:=\frac{2\pi^{d/2}}{\Gamma (d/2)}
\end{equation}
be the \textit{surface area\/} of  the \textit{$(d-1)$-dimensional unit sphere\/} $\partial \BB$ embedded in $\RR^d$.
For function $u\in \sbh_*(O)$, the {\it  Riesz measure of\/} $u$ is a  Borel 
(or Radon \cite[A.3]{R}) \textit{ positive measure }
\begin{equation}\label{df:cm}
\begin{split}
\varDelta_u&:= c_d {\bigtriangleup}  u\overset{\eqref{m+}}{\in} \Meas^+(  O), 
\quad \text{where} \\
c_d&\overset{\eqref{sd-1}}{:=}\frac{1}{s_{d-1}(1+( d-3)^+)}=\frac{\Gamma(d/2)}{2\pi^{d/2}\max \{1, d-2\bigr\}}.
\end{split}
\end{equation}
In particular,   $\varDelta_u(S)<+\infty$ for each subset $S\Subset   O$.
By definition, $\varDelta_{\boldsymbol{-\infty}}(S):=+\infty$ for each  $S\subset 
  O$. We use the  \textit{outer Hausdorff $p$-measure} $\varkappa_p$ with $p\in \NN_0$
\cite[A6]{Chirka}:
\begin{subequations}\label{df:sp}
  \begin{align}
\varkappa_p(S)&:=b_p \lim_{0<r\to 0} \inf \biggl\{\sum_{j\in \NN}r_j^p\,\colon  
  S\subset \bigcup_{j\in \NN}B(x_j,r_j),  0\leq r_j<r\biggr\}, 
\tag{\ref{df:sp}H}\label{df:spH}
 \\ 
b_p &\overset{\eqref{df:cm}}{:=}
\begin{cases} 1\quad&\text{if  $p=0$,}\\
2\quad&\text{if  $p=1$,}\\
\dfrac{s_{p-1}}{p}\quad&\text{ if $p\in \NN,$}
\end{cases}  \quad \text{is the {\it volume of the unit ball $\BB$ in  $\RR^p$}}.
\tag{\ref{df:sp}b}\label{df:spb}
\end{align}
\end{subequations}
Thus, for $p=0$, for any $S\subset \RR^d$, its Hausdorff $0$-measure $\varkappa_0(S)$ is the cardinality $\#S$ of $S$, for $p=d$ we see that $\varkappa_d\overset{\eqref{df:spH}}{=:}\lambda_d$ is  the 
 the {\it Lebesgue measure\/} on 
$S \subset \RR_{\infty}^d$, 
and $\sigma_{d-1}:=\varkappa_{d-1}\bigm|_{\partial \BB}$ is the $(d-1)$-dimensional   measure of surface area on the unit sphere $\partial \BB$ in the usual sense. 

\subsection{Topological concepts. Inward-filled hull of subset in  open set}\label{hullin} 

Let $O$ be a topological space, $S\subset O$, $x\in O$.  
We denote by $\Conn_O S$ a set of all connected components of $S$, and  
 $\conn_O (S,x) \in \Conn_O S$ is connected component of $S$ containing  $x$. We write $\clos_O S$, $\Int_O S$,  and $\partial_O S$ for  the \textit{closure,\/} the\textit{ interior,\/}  and the \textit{boundary\/} of $S$ in  $O$. 
The set $S$ is \textit{$O$-precompact\/} if $\clos_O S$ is a compact subset of $O$, and we write $S\Subset O$.  
\begin{definition}
\label{df:hole}
An arbitrary  $O$-precompact  connected component of $O\setminus S$ is called  a \textit{hole\/} in  $S$ with respect to\/ $O$. 
The union of a subset $K\subset O$ with all holes in it will be called the  \textit{inward-filled hull\/} of this set $K$ with respect to $O$ that is denoted further as 
\begin{equation}\label{inhull}
\inhull_O K:=K\bigcup \Bigl(\bigcup \{C\in \Conn_O (O\setminus K) \colon C\Subset O\}\Bigr). 
\end{equation}
Denote by $O_{\infty}$  the {\it  Alexandroff one-point compactification of\/}   $O$ with underlying set $O \sqcup \{\infty\}$, where $\sqcup$ is the \textit{disjoint union\/} of $O$ with a single point $\infty \notin O$. If this space $O$ is a topological subspace of some ambient topological space $T\supset O$, then this point $\infty$  can be identified with the boundary $\partial O\subset T$ , considered as a single point $\{\partial O\}$.
\end{definition} 
Throughout this article, we use these topological concepts only in  cases when  $O$ is an {\it open non-empty proper Greenian open subset} \cite[Ch.5, 2]{Helms} of $\RR_{\infty}^d=:T$, i.\,e.,   
\begin{subequations}\label{ODj}
\begin{align}
\varnothing \neq O=\Int_{\RR_{\infty}^d}O= \bigsqcup_{j\in N_O} D_j\neq \RR_{\infty}^d, \quad j\in N_O\subset \NN, \quad D_j=\conn_{\RR_{\infty}^d}(O,x_j), 
\tag{\ref{ODj}O}\label{{ODj}O}
\\
\intertext{where points  $x_j$  lie in {\it different connected components\/} $D_j$ of $O\subset \RR_{\infty}^d$, \underline{or}}
\quad \text{$O=D$, where
$D$ is an open connected subset, i.\,e., 
a {\it domain}}.
\tag{\ref{ODj}D}\label{{ODj}D}
\end{align}
\end{subequations}
The dependence on such an open set $O$ or such domain $D$ for constants $\const_{\dots}$ will not be indicated in the subscripts and is not discussed.
For an open set $O$ from \eqref{{ODj}O}, we often use results that are proved in our references only for domains $D$ from \eqref{{ODj}D}. This is acceptable since all such cases concern only to individual domains-components $D_j$. So, if $S\Subset O$, then $S$ meets only finite many components $D_j$. In addition, we give proofs of our statements only for cases $O,D \subset \RR^d$. If we have 
$o\notin D_j=D\ni \infty$, then we can use the  \textit{inversion\/} relative to the 
sphere $\partial B(o,1)$ centered at $o \in  \RR^d$ as in \eqref{stK}.
\begin{proposition}[{\rm \cite[6.3]{Gardiner}, \cite{Gauther_B}}]\label{KOc}
Let  $K$ be a compact  set in an open set  $O \subset \RR^d$.  Then 
\begin{enumerate}[{\rm (i)}]
\item\label{Ki} $\inhull_{O}  K$ is a compact subset in $O$;
\item\label{Kii} the set\/ $O_{\infty} \setminus \inhull_{O}  K$ is connected and locally connected subset in\/ $O_{\infty}$; 
\item\label{Kiii}   the inward-filled hull of $K$ with respect to $O$ coincides with the complement in $O_{\infty}$ of  connected component of $O_{\infty}\setminus K$ containing the point $\infty$, i.\,e., 
\begin{equation*}
\inhull_{O}  K=O_{\infty}\setminus \conn_{O_{\infty}\setminus K}(\infty);
\end{equation*}
\item\label{Kiiv} if $O'\subset \RR_{\infty}^d$ is an open subset and 
 $O\subset  O'$, then  $\inhull_{O}  K\subset \inhull_{O'}  K$;
\item\label{Kiv} $\RR^d\setminus \inhull_{O}  K$ has only finitely many components, i.\,e., $$\#\Conn_{\RR^d_{\infty}}(\RR^d\setminus \inhull_{O}  K)<\infty.$$
\end{enumerate}
\end{proposition}

\section{Gluing Theorems}\label{GT}
\setcounter{equation}{0}

\begin{gluingtheoremA}[{\rm \cite[Theorem 2.4.5]{R},\cite[Corollary 2.4.4]{Klimek}}]\label{gl:th1}
Let $\mathcal O$ be an open set in $\RR^d$, 
and let $\mathcal O_0$ be a  subset of ${\mathcal O}$. If $u\in \sbh({\mathcal O})$, $u_0\in \sbh ({\mathcal O}_0)$, 
and 
\begin{equation}\label{Uu}
\limsup_{\mathcal O_0\ni x'\to x}u_0(x')=u(x) \quad\text{for each $x\in {\mathcal O}\cap \partial {\mathcal O}_0$}, 
\end{equation}
then the formula 
\begin{equation}\label{gU}
U:=\begin{cases}
\max\{u,u_0\} &\text{ on ${\mathcal O}_0$},\\
 u &\text{ on ${\mathcal O}\setminus {\mathcal O}_0$}
\end{cases}
\end{equation}
defines a subharmonic function on ${\mathcal O}$.
\end{gluingtheoremA}
\begin{gluingtheorem}\label{gl:th2}
Let $O, O_0$ be a pair of open subsets in $\RR^d$,  
$v\in \sbh(O)$ and $ v_0\in \sbh(O_0)$ be a pair of  functions such that 
\begin{subequations}\label{g01}
\begin{align}
\limsup_{\stackrel{x'\to x}{x'\in O_0\cap O}} v(x')&\leq 
v_0(x) \quad\text{for each $x\in O_0\cap \partial O$},
\tag{\ref{g01}$_0$}\label{g010}
\\
\limsup_{\stackrel{x'\to x}{x'\in O_0\cap O}} v_0(x')&\leq v(x) \quad\text{for each $x\in O\cap \partial O_0$}.
\tag{\ref{g01}$_1$}\label{g011}
\end{align}
\end{subequations}
Then the function 
\begin{equation}\label{Vv}
V:=\begin{cases}
v_0&\text{ on $O_0\setminus O$},\\
\max \{v_0,v\}\leq v_0^++v^+&\text{ on $O_0\cap O$},\\
v&\text{ on $O\setminus O_0$,}
\end{cases}
\end{equation}
is subharmonic on $O_0\cup O$.
\end{gluingtheorem}
\begin{proof} It is enough to apply  Gluing Theorem A  twice:
\begin{enumerate}
\item[{\bf [O$_0$]}] to one pair of functions  
\begin{subequations}\label{O0}
\begin{align*}
u&:=v_0\in \sbh(O_0), \quad {\mathcal O}:=O_0;
\\
u_0&:=v\bigm|_{O\cap O_0}\in \sbh(O\cap O_0) , 
\quad {\mathcal O}_0:=O\cap O_0\subset O_0, 
\end{align*}
\end{subequations} 
under condition \eqref{g010} realizing condition \eqref{Uu};
\item[{\bf [O]}] 
to another pair of functions 
\begin{subequations}\label{OO}
\begin{align*}
u&:=v\in \sbh(O), \quad {\mathcal O}:=O;
\\
u_0&:=v_0\bigm|_{ O_0\cap O}\in \sbh(O_0\cap O) , 
\quad {\mathcal O}_0:= O_0\cap O\subset O, 
\end{align*}
\end{subequations} 
under condition \eqref{g011} realizing condition \eqref{Uu}.
\end{enumerate}
These two glued subharmonic functions coincide at the open intersection
$O\cap O_0$ and give subharmonic function $V$ defined in \eqref{Vv}.
\end{proof}

\begin{gluingtheorem}[{\rm quantitative version}]\label{gl:th3}
Let $O$ and $O_0$ be a pair of open subsets in $\RR^d$, and 
$v\in \sbh(O)$ and    $g\in \sbh(O_0)$ be a pair of  functions such that 
\begin{subequations}\label{g01v}
\begin{align}
-\infty <m_v\leq & \inf_{x\in O\cap \partial O_0} v(x), 
\tag{\ref{g01v}m}\label{g01vOm}
\\ 
 \sup_{x\in O_0\cap \partial O}\limsup_{\stackrel{x'\to x}{x'\in O_0\cap O}} v(x')&\leq M_v<+\infty, 
\tag{\ref{g01v}M}\label{g01vOM}
\\
-\infty < \sup_{x\in O\cap \partial O_0}
\limsup_{\stackrel{x'\to x}{x'\in O\cap O_0}} g(x')\leq m_g&
< M_g\leq  \inf_{x\in O_0\cap \partial O} g(x) <+\infty.
\tag{\ref{g01v}g}\label{g01g}
\end{align}
\end{subequations}
If we choose the function 
\begin{equation}\label{v0g}
v_0:=\frac{M_v^++m_v^-}{M_g-m_g} (2g-M_g-m_g)\in \sbh(O_0), 
\end{equation} 
then the function $V$ from \eqref{Vv} is subharmonic on $O_0\cup O$. 
\end{gluingtheorem}
\begin{proof}
The function $v_0$ from definition \eqref{v0g} is subharmonic on $O_0$ since this function $v_0$  has a form $\const^+g+\const$ with $\const^+\in \RR^+$, $\const \in \RR$. 
In addition, by construction \eqref{v0g}, for each $x\in O_0\cap \partial O$, we obtain
\begin{multline*}
\limsup_{\stackrel{y\to x}{y\in O_0\cap O}} v(y)\overset{\eqref{g01vOm}-\eqref{g01vOM}}{\leq}
 M_v^++m_v^-= \frac{M_v^++m_v^-}{M_g-m_g} (2M_g-M_g-m_g)\\
 \overset{\eqref{g01g}}{\leq}\frac{M_v^++m_v^-}{M_g-m_g} \Bigl(2 \inf_{x\in O_0\cap \partial O} g(x)  -M_g-m_g\Bigr)
=\inf_{x\in O_0\cap \partial O} \frac{M_v^++m_v^-}{M_g-m_g} \Bigl(2  g(x)  -M_g-m_g\Bigr)\\
\overset{\eqref{v0g}}{=} \inf_{O_0\cap\partial O} v_0\leq 
v_0(x), \quad \forall x \in O_0\cap \partial O.
\end{multline*} 
Thus,  we have  \eqref{g010}. 
Besides, by construction \eqref{v0g}, for each $x\in O\cap \partial O_0$, we obtain
\begin{multline*}
\limsup_{\stackrel{x'\to x}{x'\in O_0\cap O}} v_0(x')
\overset{\eqref{v0g}}{\leq}
\frac{M_v^++m_v^-}{M_g-m_g} \biggl(2\limsup_{\stackrel{x'\to x}{x'\in O_0\cap O}} g(x')-M_g-m_g\biggr)
\\
\overset{\eqref{g01g}}{\leq}
\frac{M_v^++m_v^-}{M_g-m_g} (2m_g-M_g-m_g)
=-(M_v^++m_v^-)\leq  -m_v^-\leq m_v
\\
\overset{\eqref{g01vOm}}{\leq}
 \inf_{x\in O\cap \partial O_0} v(x)\leq v(x), 
 \quad \forall x\in  O\cap \partial O_0.
\end{multline*} 
Thus,  we have  \eqref{g011}, and  Gluing Theorem \ref{gl:th3} follows from 
Gluing Theorem \ref{gl:th2}. 
\end{proof}

\begin{remark} 
Theorems of this section can be easily transferred to the cone of plurisubharmonic functions
\cite[Corollary 2.9.15]{Klimek}.
We sought to formulate our theorems and their proofs with the possibility of their fast transport to the plurisubharmonic functions and to abstract potential theories with  more general constructions based on the theories of harmonic spaces and sheaves  in the spirit 
of books \cite{BB66},  \cite{ConCor}, \cite{BB80}, \cite{BBC81}, \cite{BH}, \cite{LNMS}, \cite{AL},   etc.
\end{remark}

\section{Gluing with Green's Function}\label{GTg}
\setcounter{equation}{0}

\begin{definition}[{\rm \cite{R}, \cite{HK}, \cite{Landkoff}}]\label{df:kK} 
 For $q\in \RR$ and $d\in \NN$, we set  
\begin{subequations}\label{kK}
\begin{align}
k_q(t)& := \begin{cases}
\ln t  &\text{ if $q=0$},\\
 -\sgn (q)  t^{-q} &\text{ if $q\in \RR_*$,} 
\end{cases}
\qquad  t\in \RR_*^+,
\tag{\ref{kK}k}\label{{kK}k}
\\
K_{d-2}(x,y)&:=\begin{cases}
k_{d-2}\bigl(|x-y|\bigr)  &\text{ if $x\neq y$},\\
 -\infty &\text{ if $x=y$ and $d\geq 2$},\\
0 &\text{ if $x=y$ and  $d=1$},\\
\end{cases}
\quad  (x,y) \in \RR^d\times \RR^d.
\tag{\ref{kK}K}\label{{kK}K}
\end{align}
\end{subequations}
 \end{definition}

Recall that a set $E\subset \RR^d$ is called {\it polar\/} if there is a function $u\in \sbh_*(\RR^d)$ such that 
\begin{equation}\label{E}
\Bigl(E\subset (-\infty)_u:=\{ x\in \RR^d\colon u(x)=-\infty \} \Bigr)
\Longleftrightarrow \Bigl(\text{Cap}^* E=0\Bigr),
\end{equation}
where the set $(-\infty)_u$ is {\it minus-infinity\/} $G_{\delta}$-set for the function 
$u$,   
\begin{equation*}
\text{Cap}^*(S):=
\inf_{S\subset O=\Int O}  
\sup_{\stackrel{C=\clos C\Subset O}{\mu\in \Meas^{1+}(C)}} 
 k_{d-2}^{-1}\left(\iint K_{d-2} (x,y)\dd \mu (x) \dd \mu(y) \right)
\end{equation*}
is the {\it outer capacity\/} of $S\subset \RR^d$.

Let $\mathcal O$ be an \textit{open  proper} subset in $\RR_{\infty}^d$. 
Consider a point $o\in \RR^d$ and subsets $S_o, S\subset \RR_{\infty}^d $  such that  
\begin{equation}\label{x0S}
\RR^d \ni o \in \Int S_o\subset S_o\Subset S \subset \Int \mathcal O
=\mathcal O\subset \RR^d_{\infty}\neq \mathcal O. 
\end{equation} 
Let $D$ be a \textit{domain\/} in $\RR_{\infty}^d$ such that
\begin{equation}\label{Dg}
 o\overset{\eqref{x0S}}{\in} \Int  S_o\subset S_o\Subset D \Subset S \subset \mathcal O.
\end{equation}
Such domain $D$ possesses the extended  Green's function $g_D (\cdot, o)$\/  {\rm (see \cite[5.7.2]{HK}, \cite[Ch.~5, 2]{Helms})} with pole at the point 
$o\overset{\eqref{Dg}}{\in} D$ 
 described by the following properties:  
\begin{subequations}\label{gD}
\begin{align}
g_D(\cdot , o)&\in \sbh^+ \bigl(\RR_{\infty}^d\setminus \{o\}\bigr)\subset \sbh^+\bigl(\mathcal O\setminus \{o\}\bigr) , 
\tag{\ref{gD}s}\label{gDs}\\
g_D(\cdot ,o)&= 0\text{ on $\RR_{\infty}^d\setminus \clos D \supset \mathcal O\setminus \clos D\supset \mathcal O\setminus S$}, 
\tag{\ref{gD}$_0$}\label{gD0}\\
g_D(\cdot , o)&\in \har \bigl(D\setminus \{o\}\bigr)\subset 
\har\bigl(S_o\setminus\{o\}\bigr)\overset{\eqref{Bo}}{\subset} \har\bigl(B_{\circ}(o,r_o)\bigr)
\tag{\ref{gD}h}\label{gDh}\\
\intertext{with a number $r_o\in \RR_*^+$, $g_D(o,o):=+\infty$,}
g_D(x,o)&\overset{\eqref{{kK}K}}{=}-K_{d-2}(x,o)+O(1) \quad\text{as $o\neq x\to o$}.
\tag{\ref{gD}o}\label{gD0a}\\
\intertext{and the following strictly positive number}
0<M_g&:=\inf_{x\in \partial S_o} g_D (x,o)=
\const^+_{o, S_o, D,S} 
 \tag{\ref{gD}M}\label{Mg}\\
\intertext{depends only on $S_o,S,D$ and the pole $o$, and, by the minimum principle, we have}
g_D(x,o)-M_g&\overset{\eqref{Mg}}{\geq} 0\quad\text{for all $x\in S_o\setminus \{o\}$}.
\tag{\ref{gD}M+}\label{Mg+}
\end{align}
\end{subequations}
Properties \eqref{gD} for the extended Green's function 
$g_D (\cdot, o)$\/  from \eqref{x0S}--\eqref{Dg} are well known \cite[4.4]{R}, \cite[5.7]{HK}, and 
property \eqref{x0S} follows from $0< g(\cdot ,o)\in C\bigl(D\setminus \{o\}\bigr)$ on $D\setminus \{o\}$.  

\begin{gluingtheorem}\label{gl:th4}
Under conditions  \eqref{x0S}
suppose that a function $v \in \sbh(\mathcal O\setminus S_o)$ satisfies constraints from above and below  in the form
\begin{equation}\label{vabS}
-\infty<m_v\overset{\eqref{g01vOm}}{\leq} \inf_{S\setminus S_o} v\leq  
\sup_{S\setminus  S_o} v\overset{\eqref{g01vOM}}{\leq}  M_v<+\infty. 
\end{equation}
Every domain $D$ with inclusions  \eqref{Dg} possesses  the extended Green's function $g_D(\cdot , o)$ with pole $o\in \Int S_o$, properties  \eqref{gD} and constant $M_g$ of   \eqref{Mg} such that the choice of function 
\begin{subequations}\label{gDV}
\begin{align}
v_0&\overset{\eqref{v0g}}{:=}\frac{M_v^++m_v^-}{M_g} \bigl(2g_D(\cdot , o) -M_g\bigr)\in 
\sbh\bigl(\RR_{\infty}^d\setminus \{o\}\bigr)\subset 
\sbh\bigl(\Int S\setminus \{o\}\bigr),
\tag{\ref{gDV}v}\label{gDVv}
\\
\intertext{defines the subharmonic function}
V&\overset{\eqref{Vv}}{:=}\begin{cases}
v_0&\text{ on $S_o$},\\
\sup \{v_0,v\}\leq v_0^++v^+&\text{ on $S\setminus S_o$},\\
v&\text{ on $\mathcal O\setminus S$,}
\end{cases} \qquad \text{from $\sbh_*\bigl(\mathcal O\setminus \{o\}\bigr)$,}
\tag{\ref{gDV}V}\label{gDVV}\\
\intertext{satisfying the conditions}
 V&\overset{\eqref{gDh}}{\in} \har \bigl(S_o\setminus \{o\}\bigr)\overset{\eqref{Bo}}{\subset} \har\bigl(B_{\circ}(o,r_o)\bigr)
\quad\text{with a number $r_o\in \RR_*^+$},
\tag{\ref{gDV}h}\label{gDhV}\\
v(x)&\overset{\eqref{gDVV}}{\leq} V(x)\overset{\eqref{gDVv}}{\leq}
 M_v^++ 2\frac{M_v^++m_v^-}{M_g} g_{D} (x,o)
	 \quad\text{for all $x\in  S\setminus S_o$,}
\tag{\ref{gDV}+}\label{gDhV+}\\
0& \leq V(x)\overset{\eqref{Mg+}}{\leq}
 2\frac{M_v^++m_v^-}{M_g} g_{D} (x,o)
\quad\text{for all $x\in  S_o\setminus \{o\}$,}
\tag{\ref{gDV}$_0^+$}\label{gDhV++}\\
V(x)&\overset{\eqref{gD0a}}{=}-2\frac{M_v^++m_v^-}{M_g} K_{d-2} (x,o)+O(1) \quad\text{as $o\neq x\to o$}.
\tag{\ref{gDV}o}\label{gD0aV}
\end{align}
\end{subequations}
\end{gluingtheorem}
\begin{proof} It is enough to apply Gluing Theorem \ref{gl:th3}
 with 
\begin{equation*}
O:=\mathcal O\setminus \clos S_o, 
\quad O_0:=\Int S\setminus \{o\}, \quad 
g:=g_D(\cdot,o), \quad  m_g:=0 
\end{equation*} 
in accordance with the reference marks indicated over relationships in 
 \eqref{vabS}--\eqref{gDV}.
\end{proof}

\begin{remark}\label{MgS} The choice of the domain $D$, and hence the constant $M_g$ in \eqref{Mg} and \eqref{gD0aV}, is entirely determined by the mutual arrangement of the sets $S_o\Subset S$.
\end{remark}

Given $S\subset \RR^d$ and $r\in \RR_*^+$, a set 
\begin{equation}\label{Scup}
S^{\cup r}\overset{\eqref{{B}B}}{:=}\bigcup\limits_{x\in S} B(x,r).
\end{equation}
is called a {\it outer $r$-parallel open set\/ {\rm \cite[Ch.~I,\S~4]{Santalo}}  for\/} $S$. 
Easy to install the following

\begin{proposition}\label{llemmaD} 
Let a subset $S_o\Subset \RR^d$
 be connected, and $r\in \RR_*^+$. Then $S_o^{\cup r}$, $S_o^{\cup (2r)}$,
$S_o^{\cup (3r)}$ are domains,  and there is a  Dirichlet regular domain $D_r\subset S_o^{\cup (3r)}$ such that
\begin{equation}\label{SDS}
S_o^{\cup r}\Subset D_r\Subset S_o^{\cup (2r)}. 
\end{equation}

\end{proposition}

For $v\in L^1\bigl(\partial B(x,r)\bigr)$, we define the averaging value of $v$ at the point $x$ on the sphere $\partial B(x,r)$ as
\begin{equation}\label{v0}
v^{\circ r}(x)\overset{\eqref{sd-1}}{:=}\frac{1}{s_{d-1}}\int_{\partial \BB} v (x+rs)\dd \sigma_{d-1}(s),
\end{equation}
 where $\sigma_{d-1}$   measure of surface area on the unit sphere $\partial \BB$.

\begin{gluingtheorem}\label{gl:th_es}
Let $\mathcal O\subset \RR^d$ be an open subset, and  $S_o\subset \RR^d$ be a connected set such that there is a point
\begin{equation}\label{S0}
 o\overset{\eqref{x0S}}{\in} \Int  S_o\subset S_o \Subset \mathcal O.
\end{equation}
Let  $r\in \RR^+$ be a  number such that
\begin{equation}\label{posr}
0<3r<\dist(S_o, \partial \mathcal O),
\end{equation}
and $D_r$ be a domain from Proposition\/ {\rm \ref{llemmaD}} satisfying \eqref{SDS}.
Let $v \in \sbh_*(\mathcal O\setminus S_o)$ be a function satisfying constraints from above and below  in the form
\begin{subequations}\label{avv}
\begin{align}  
v&\leq M_v <+\infty \quad\text{on $S_o^{\cup (3r)}\setminus S_o$},
\tag{\ref{avv}M}\label{avvM}
\\
m_v&:=\inf \bigl\{ v^{\circ r}(x)\colon x\in S_o^{\cup (2r)}\setminus S_o^{\cup (r)}\bigr\}.
\tag{\ref{avv}m}\label{avvm}
\end{align}
\end{subequations}
Then $m_v>-\infty $, and   there is a function $V\in \sbh_*(\mathcal{O}\setminus \{o\}) $ satisfying \eqref{gDhV}--\eqref{gDhV+}, i.\,e., 
\begin{subequations}\label{VK}
\begin{align}
0<V&\in \har^+\bigl(S_o\setminus \{o\}\bigr) \quad\text{on $S_o\setminus \{o\}$},
\tag{\ref{VK}h}\label{gDVh}\\
V&=v \quad \text{on $\mathcal O\setminus S_o^{\cup (3r)}$},
\tag{\ref{VK}=}\label{gDV=}\\
v(x)\leq V(x)&\leq M_v^++2\frac{M_v^++m_v^-}{M_g} g_{D_r}(x,o) \quad \text{for all $x\in  S^{\cup (3r)}\setminus S_o$},
\tag{\ref{VK}+}\label{gDVleq}
\\
0< V(x)&\leq 2\frac{M_v^++m_v^-}{M_g} g_{D_r}(x,o) \quad \text{for all $x\in  S_o\setminus \{o\}$},
\tag{\ref{VK}$_0^+$}\label{gDVleq+}
\\
\intertext{and such that}
V(x)&\overset{\eqref{gD0aV}}{=}- 2\frac{M_v^++m_v^-}{M_g}K_{d-2} (x,o)+O(1) \quad\text{as\quad  $o\neq x\to o$}
\tag{\ref{VK}o}\label{gDVo},\\
\intertext{where}
0<M_g&:=\inf_{x\in \partial S_o^{\cup r}} g_{D_r} (x,o)=
\const^+_{o, S_o, r,{D_r}}=\const^+_{o,S_o,r}
\tag{\ref{VK}g}\label{gDVg}
\end{align}
\end{subequations}

\end{gluingtheorem}
\begin{proof}
 We have $m_v>-\infty$ since the function  
$v^{\circ r}$  is continuous in $ S_o^{\cup (2r)}\setminus S_o^{\cup r}$
 \cite[Theorem 1.14]{Helms}. 
The function  $v$ can be transformed using the Perron\,--\,Wiener\,--\,Brelot method (into the open ``layer'' $S_o^{\cup (3r)}\setminus \clos S_o$ from 
boundary of this layer) to a new \textit{subharmonic function $\widetilde v\geq v$ on} $\mathcal O \setminus S_o$ such that $\widetilde v\in \har (S_o^{\cup (3r)}\setminus \clos S_o)$ and 
$\widetilde v=v$ on $\mathcal O\setminus S_o^{\cup (3r)}$.
It follows from the principle of subordination (domination) for harmonic continuations
and  the  maximum principle that 
\begin{equation}\label{estwv}
-\infty <m_v\leq \widetilde v\quad\text{on $S_o^{\cup (2r)}\setminus S_o^{\cup r}$}, \quad
\quad \widetilde v\leq M_v\quad\text{on $S_o^{\cup (3r)}\setminus S_o$}.
\end{equation}
If we choose in Gluing  Theorem \ref{gl:th4} for the role a set $S_o$ the set $S_o^{\cup r}$, and instead of $S$ the set $S_o^{\cup (2r)}$, then, by construction \eqref{gDVv}--\eqref{gDVV} and conditions \eqref{gDhV}--\eqref{gD0aV} , we get 
series of conclusions \eqref{VK} of Theorem \ref{gl:th_es} in view of \eqref{estwv} since 
\eqref{estwv} gives \eqref{vabS} for $\widetilde v$ instead of $v$.

The possibility of replacing a constant $\const^+_{o, S_o, r,{D_r}}$ with $\const^+_{o,S_o,r}$ depending only on $o,S_o,r$  in\eqref{gDVg} follows from the 
Remark \ref{MgS}. 
\end{proof}

\section{Linear  balayage of charges and measures}\label{Ssec_balm}
\setcounter{equation}{0}

In this section \ref{Ssec_balm} we discuss conventional linear  balayage that is particular case of  \eqref{b0}. Next, we call linear balayage simply balayage.

\begin{definition}\label{df:1} Let $\vartheta, \mu  \in \Meas(S)$, $S\subset \Borel (\RR^d_{\infty})$.  
 Let $H\subset \RR_{\pm\infty}^{S}$ be a class of Borel-me\-a\-s\-u\-r\-a\-ble functions on $S$. 
 Let us assume that the integrals $\int h \dd{\vartheta}$ and  $\int h \dd{\mu}$ are well defined with values in $\RR_{\pm\infty}$ for each function $h\in H$. We write ${\vartheta} \preceq_H \mu$ and say that the charge  $\mu$ is a {\it balayage,\/} or, sweeping (out), of the charge ${\vartheta}$ {\it for\/} $H$, or, briefly, $\mu$ is a  $H$-balayage of $\vartheta$,   if 
\begin{equation}\label{balnumu}
\int h \dd {\vartheta} \overset{\eqref{b0}}{\leq} \int h\dd \mu \quad\text{\it for all\/ $h\in H$.}
\end{equation} 
The relation $\preceq_H$ is a \textit{preorder}  on a part of $\mathcal M_H\subset \Meas(S)$, where the integrals $\int h\dd \mu$ are well defined for all $\mu \in \mathcal M_H$. If integrals $\int h\dd \mu$ are finite for each $h\in H$ and $\mu \in \mathcal M_H$, and  
\begin{equation}\label{H}
H=-H,
\end{equation} 
then this relation is \textit{symmetric} on $\mathcal M_H$, i.\,e., the inequality in  \eqref{balnumu} becomes the equality
\begin{equation}\label{balnumu=}
\int h \dd {\vartheta} \overset{\eqref{H}}{=} \int h\dd \mu \quad\text{\it for all\/ $h\in H$.}
\end{equation}
\end{definition}

 In this article, we consider only balayage  for 
$H\subset \sbh(S)\subset \RR_{-\infty}^S$.
In this case,  the integrals from \eqref{balnumu} are well  defined for all measures 
${\vartheta}, \mu \in \Meas_{\comp}^+(S)$ with values in $\RR_{-\infty}$,
and for all  absolutely continuous with respect to the  Lebesgue measure $\lambda_d$ charges  
${\vartheta},\mu\overset{\eqref{MLl}}{\in} L_{\loc}^1(S)\dd \lambda_d$ with values in $\RR$, etc.

\begin{proposition}\label{Prtr}  Let  $O\subset \RR^d$ be an open set, $\mu \in \Meas(O)$ be a  $H$-balayage of ${\vartheta}\in  \Meas(O)$, $O'\subset \RR^d$ be an open set, and $H'\subset \RR_{\pm\infty}^{O'}$. 
\begin{enumerate}[{\rm 1.}]
\item\label{b1} If  $1\in H$, then ${\vartheta} (O)\leq \mu(O)$. 
\item\label{b2}  If $\pm 1\in H$, then ${\vartheta} (O)= \mu(O)$.
\item\label{b3} If $H'\subset H$, then  $\mu$  is a $H'$-balayage of ${\vartheta}$.
 \item If  $O'\subset O$  and $supp \vartheta\cup \supp \mu\subset O'$, then $\mu\bigm|_{O'}$  is a balayage of ${\vartheta}\bigm|_{O'}$ for  $H\bigm|_{O'}$.
 \end{enumerate}
\end{proposition}
All statements of Proposition \ref{Prtr} are obvious.

\begin{remark} Balayage of  charges and measures with a non-compact support is also occur frequently and are used  in Analysis. So, a bounded domain $D\subset \RR^d$ is called a {\it quadrature domain\/} (for harmonic functions)  if there is a charge $\mu \in \Meas_{\comp} (D)$ such that the restriction $\lambda_d\bigm|_D$  is a balayage of $\mu$ for the class $\har (D)\cap L^1(D)$. In connection with the quadrature domains, see very informative overview  \cite[3]{quad} and bibliography in it.
\end{remark}

\begin{proposition}\label{pr:diff} If $\mu \in \Meas^+_{\comp}(O)$   is a balayage of $\vartheta\in \Meas_{\comp}^+(O)$ for $\sbh(O)\cap C^{\infty}(O)$, then  $\mu$   is a balayage of $\vartheta$ for $\sbh(O)$. 
 \end{proposition} 
Proposition \ref{pr:diff} follows from \cite[Ch. 4, 10, Approximation Theorem]{Doob}.

\begin{example}[{\rm \cite[3]{Gamelin}, \cite{Koosis}, \cite{Koosis96}, \cite{C-R}, \cite{C-RJ}, \cite{Schachermayer}, \cite{HN}, \cite{Kha91}--\cite{KhaKha19}, \cite{KhaTalKha15}, \cite{BaiTalKha16}}]\label{sbhJ} 
If a measure $\mu\in \Meas_{\comp}^+(O)$  is  a balayage of the Dirac measure $\delta_{x}$ for $\sbh(O)$, where $x\in O$, then this measure $\mu$ is called a {\it Jensen measure for\/} $x$. The class of such measures is denoted by $J_x(O)$. 

\end{example}

\begin{example}\label{sbhom} 
 We denote by   $\omega_D\colon D\times{\Borel}(\partial D)\to [0,1]$ the \textit{harmonic measure for\/} $D$ with non-polar boundary $\partial D\subset \RR_{\infty}^d$. Measures $\omega_D(x,\cdot)\overset{\eqref{m1}}{\in} \Meas^{1+}(\partial D)$, $x\in D$, will also be called a harmonic measure  for $D$, but with specification,  \textit{at\/} $x\in D$.  If $D\Subset O$, then measures 
\begin{equation}\label{om}
a\delta_x+b\omega_D(x,\cdot)\in J_x(O), \quad a,b\in \RR^+, \quad a+b=1.
\end{equation}
Likewise, if \begin{equation*}
\sum_{k\in \NN}b_k=1,\quad  b_k \in \RR^+, \quad D_k, D:=\bigcup_kD_k\Subset O
\quad\text{ are Greenian, then $\sum_k b_k\omega_{D_k}(x,\cdot)\in J_x(O)$.}  
\end{equation*}
So, the surface measure  $\sigma_{d-1}$ in the unit sphere $\partial \BB$ belongs to  $J_0(r\BB)$ for each  $r>1$. 
\end{example}

\begin{example}\label{aJ} 
Useful examples of Jensen measures from $J_0(B(r))$ are \textit{probability\/} measures
\begin{equation}\label{eps}
 \alpha_r^{\infty}\overset{\eqref{MLl}}{\in} \left(C_0^{\infty}(r\BB)\right)^+\dd \lambda_d\overset{\eqref{m1}}{\in} \Meas^{1+}\bigl(B(r)\bigr),\; \alpha_r^{\infty}(S)=\alpha_1^{\infty}( S/r), \; S\in \Borel(\RR^d), 
\end{equation}
$ r\in \RR_*^+$,
invariant under the action of the orthogonal group on $\RR^d$.
\end{example}

\begin{example}[{\rm \cite[3]{Gamelin}, \cite{Kha96}, \cite{Kha03}, \cite{Kha01II}, \cite{Kha07}, \cite{KudKha09}}]\label{sbhAS} Let $x\in O$.   If $\mu\in \Meas_{\comp}^+(O)$ is a balayage of $\delta_{x}$ for $\har(O)$, then such measure $\mu$ is called an {\it Arens\,--\,Singer  measure for\/} $x$. The class of such measures is denoted by $AS_x(O)\supset J_x(O)$. Arens\,--\,Singer measures are often referred to as representing measures.
\end{example}

\begin{proposition}\label{pr:ii} For  $H\subset \sbh(O)$ (resp.,
$H =-H\subset \har(O)$), let a measure $\mu\in \Meas_{\comp}(O)$ be a  balayage of a measure  $\vartheta\in \Meas_{\comp}$ for $H$. 

Let $\iota_x\in J_x(O)\overset{\eqref{m1}}{\subset} \Meas^{1+}_{\comp}(O) $ 
(resp., $\iota_x\in AS_x(O)\subset \Meas^{1+}_{\comp}(O) $)
with 
\begin{equation}\label{sxi}
s_x:=\supp \iota_x\Subset B\Bigl (x,\frac{1}{2}\dist (\supp \mu, \partial O)\Bigr) \quad\text{for all $x\in \supp \mu$}
\end{equation} 
be a family Jensen  (resp., Arens\,--\,Singer) measures for points $x\in O$. The measure $\mu$  and probability measures $\iota_x$ are bounded  in aggregate, and we can to define the integral  of $\iota_x$ with respect to $\mu$ {\rm \cite{Landkoff}, \cite{Bourbaki}}
\begin{equation}\label{iimu}
\beta:=\int \iota_x\dd\mu(x)\colon h\longmapsto \int\Bigl(\int h\dd\iota_x\Bigr ) \dd \mu(x).
\end{equation}   
In particular, if every Jensen (resp., Arens\,--\,Singer) measure $\iota_x$ is a parallel shift to a point $x\in O$ of the same Jensen (resp., Arens\,--\,Singer) measure $\iota_0 $ for $0$ with the diameter $\diam \supp \iota_0$ of $\supp \iota_0$ fewer than
$\frac{1}{2}\dist (\supp \mu, \partial O)\bigr)$, then  integral $\beta$ from \eqref{iimu} 
is a classical  convolution $\beta =\iota_0*\mu$ of two measures $\iota_0$ and $\mu$:
\begin{equation}\label{sv}
\beta:=\iota_0 \ast \mu =\mu\ast \iota_0 \colon h\longmapsto \iint h (x+y) \dd \iota_0 \dd \mu,
\quad s_x\overset{\eqref{sxi}}{=}x+\supp \iota_0.
\end{equation}
In these cases both  measures $\beta$ from \eqref{iimu}--\eqref{sv} also a balayage of $\vartheta$ for $H$  with 
\begin{equation*}
\supp \beta\subset  \clos \Bigl((\supp \mu) \bigcup \bigcup_{x\in \supp \mu}s_x\Bigr)
\Subset O. 
\end{equation*}
\end{proposition}
\begin{proof} Under condition \eqref{sxi}, for subharmonic function $h\in H\subset \sbh(O)$, we have
\begin{equation}\label{ins}
\int h\dd \vartheta \leq \int h\dd \mu \leq 
\int_{\supp \mu }\Bigl(\int_{s_x}h \dd \iota_x\Bigr)\dd \mu (x)
\overset{\eqref{iimu}}{=}\int h \dd \beta 
\end{equation}
by definitions  \eqref{iimu}--\eqref{sv}. For $H\subset \har (O)$ and $\iota_x\in AS_x$, by analogy with \eqref{ins}, we have equalities in \eqref{ins}.
\end{proof}
\begin{remark}\label{remgl}
If we choose parallel shifts to $x$ of measures (Example \ref{aJ}, \eqref{eps}) 
$$\alpha_{r(x)}^{\infty}
\in \left(C_0^{\infty}(r(x)\BB)\right)^+\dd \lambda_d\in \Meas^{1+}B(r(x))
$$ as measures $\iota_x$ for  Proposition  \ref{pr:ii} with a function $r\in C^{\infty}(O)$ and with condition  \eqref{sxi}, then both measures $\beta$ from   \eqref{iimu}--\eqref{sv}  belong to the class $C_0^{\infty}(O)\dd \lambda_d$ and still $\vartheta \preceq_{H} \beta $, i.e., 
the measure $\beta \in \Meas_{\comp}(O) $ is a balayage of the measure $\vartheta$ for $H$.
\end{remark}

\begin{proposition}\label{Pr:munuh}  Let  $\mu \in \Meas_{\comp}(O)$ be a  balayage of ${\vartheta}\in  \Meas_{\comp}(O)$ for $\har (O)$.  Then 
\begin{equation}\label{bh}
\int h \dd {\vartheta}=\int h\dd \mu \quad \text{for every $h \in  \har \bigl({\inhull_O (\supp \mu \cup \supp {\vartheta}})\bigr)$}.
\end{equation}
\end{proposition}
See Subsec. \ref{hullin}, Definition \ref{df:hole} 
of inward-filled hull of compact subset $\supp \mu \cup \supp {\vartheta}$
 in $O$.
\begin{proof}  We set 
\begin{equation}\label{K}
K:=\supp \vartheta \cup \supp \mu \Subset O.
\end{equation} 
By Proposition \ref{KOc}\eqref{Kii} and \cite[Theorem 1.7]{Gardiner}, if $h\in \har\bigl({\inhull_O K}\bigr)$, then there are functions $h_k\in \har (O)$, $k\in \NN$, such that the sequence 
 $(h_k)_{k\in \NN}$ converges to $h$ in $C \bigl( {\inhull_O K}\bigr)$, and 
\begin{multline*}
\int_{{\inhull_O K}} h \dd {\vartheta}= \int_{{\inhull_O K}}  \lim_{k\to \infty}
h_k \dd {\vartheta} = \lim_{k\to \infty} \int_{O} h_k \dd {\vartheta}
\\
\overset{\eqref{balnumu}}{\leq} 
\lim_{k\to \infty} \int_{O} h_k \dd \mu= 
 \int_{{\inhull_O K}} \lim_{k\to \infty} h_k \dd \mu=
 \int_{{\inhull_O K}} h \dd {\vartheta}.
\end{multline*}
Using the opposite function $-h\in  \har\bigl({\inhull_O K}\bigr)$, we have the inverse inequality.  
\end{proof}

\begin{proposition}\label{pr:vstm} Suppose  that measures  $\varsigma, \vartheta, \mu \in \Meas_{\comp}^+(O)$ satisfy the conditions
\begin{equation}\label{var}
\begin{cases}
\varsigma &\preceq_{\har(O)}\vartheta,\\ 
\varsigma &\preceq_{\sbh(O)} \mu,
\end{cases}
\qquad \text{and }\inhull (\supp \varsigma \cup \supp \vartheta )\subset O',
\end{equation}
where $O'\Subset O$ is an open subset such that $O'\cap \supp \mu=\varnothing$. Then
$\vartheta\preceq_{\sbh(O)} \mu$.	
\end{proposition}
\begin{proof}  According to representation \eqref{{ODj}O}, it suffices to consider the case when $D:=O$ and $D':=O'$ are domains. There exists a regular for the Dirichlet problem domain $D''$ such that 
\begin{equation}\label{inss}
\inhull (\supp \varsigma \cup \supp \vartheta ) \overset{\eqref{var}}{\subset} D''\Subset D'\subset D
\end{equation}
since $\inhull (\supp \varsigma \cup \supp \vartheta )$ is compact subset in $D'$ by Proposition \ref{KOc}\eqref{Ki}.

Let $u\in \sbh_*(D)$. Then we can build a new subharmonic function $\widetilde u \in \sbh_*(D)$ such that \begin{equation}\label{tu}
\widetilde u\bigm|_{D''}\in \har (D''), \quad \widetilde u=u\quad\text{on $D\setminus D''$}, \quad 
u\leq \widetilde u\quad \text{on $D$}.
\end{equation} 
By Proposition \ref{Pr:munuh}, in view of the inclusion in \eqref{inss}, we have
\begin{equation}\label{chain}
\int_{D} u\dd \vartheta 
\overset{\eqref{var}}{=}\int_{D''}  u\dd \vartheta
\overset{\eqref{tu}}{\leq}\int_{D''} \widetilde u\dd \vartheta
\overset{\eqref{bh},\eqref{inss}}{=} \int_{D''} \widetilde u\dd \varsigma=
\int_{D} \widetilde u\dd \varsigma
\overset{\eqref{var}}{\leq} \int_{D} \widetilde{u} \dd \mu.
\end{equation}
Since $\supp \mu \subset D\setminus D' $, we can continue this chain of (in)equalities \eqref{chain} as 
\begin{equation*}
\int_{D} u\dd \vartheta \overset{\eqref{chain}}{\leq} \int_{D} \widetilde{u} \dd \mu
=\int_{D\setminus D'} \widetilde{u} \dd \mu
\overset{\eqref{tu}}{=}\int_{D\setminus D'} u  \dd \mu=\int_{D} u  \dd \mu.
\end{equation*}

\end{proof}	

\begin{proposition} Let $\mu\in \Meas^+_{\comp}(O)$ be a  $\har(O)$-balayage of  measure $\vartheta \in \Meas^+_{\comp}(O)$, and  $\varsigma\in \Meas^+_{\comp}(O)$ also  be a  $\har(O)$-balayage  of the same measure $\vartheta$. If 
\begin{equation*}
\inhull_{O}(\supp \vartheta \cup \supp \varsigma)\subset \inhull_{O}(\supp \vartheta \cup \supp \mu),
\end{equation*}
then the measure $\mu$ is a $\har (O)$-balayage of the measure $\varsigma$.
\end{proposition}
We omit the proof of this easy corollary of Proposition \ref{Pr:munuh}.

\begin{proposition}\label{pr:4}  Let  $\mu \in \Meas_{\comp}(O)$ be a  balayage of ${\vartheta}\in  \Meas_{\comp}(O)$ for $\sbh (O)$.  Then 
\begin{equation}\label{bhs}
\int  u \dd {\vartheta}\leq \int  u\dd \mu \quad \text{ for each  $u\in \sbh \bigl({\inhull_O (\supp \mu \cup \supp {\vartheta})}\bigr)$},
\end{equation}
i.\,e.,  if $O'\supset\inhull_O (\supp \mu \cup \supp {\vartheta})$ is an open set, then $\mu$ is a    $\sbh(O')$-balayage of $\vartheta$. 
\end{proposition}
\begin{proof} We use  the notation \eqref{K}. By Proposition \ref{KOc}\eqref{Kii},  if $ u\in \sbh\bigl({\inhull_O K}\bigr)$, then there is  a function $ U \in \sbh \bigl(O)$ such that $u=U$ on  $ {\inhull_O K}$ \cite[Theorem 6.1]{Gardiner},
\cite[Theorem 1]{Gauther_C}, \cite[Theorem 16]{Gauther_B}, and we have
\begin{equation*}
\int_{{\inhull_O K}}  u \dd {\vartheta}= \int_{{\inhull_O K}}  
U \dd {\vartheta} = \int_{O} U \dd {\vartheta}
\overset{\eqref{balnumu}}{\leq} 
 \int_{O} U \dd \mu= 
 \int_{{\inhull_O K}} U \dd \mu=
 \int_{{\inhull_O K}}  u \dd {\mu}
\end{equation*}
that   gives \eqref{bhs}.
\end{proof}

\begin{proposition}\label{Pr_pol}
If  $\mu\in \Meas_{\comp}^+(O)$ is  a\/  $\sbh(O)$-balayage of a measure ${\vartheta} \in \Meas_{\comp}^+(O)$, and  a set $E$ is polar, then  $\mu (O\cap E\setminus \supp {\vartheta})=0$.
\end{proposition}
This Proposition \ref{Pr_pol} is apparently well-known fact, but nevertheless we give its full 
\begin{proof} There is $k_0\in \NN$ such that $B(x,1/k_0)\Subset O$ for all 
$x\in \supp {\vartheta}$.  For any $k\in k_0+\NN_0$ 
there exists an finite cover of $\supp {\vartheta}$ by balls $B(x_j,1/k)\Subset O$ such that the open subsets 
\begin{equation*}
O_k:=\bigcup_j B(x_j,1/k)\Subset O,\quad
 \supp {\vartheta} \Subset O_k\supset O_{k+1}, \quad k\in k_0 +\NN_0, \quad \supp {\vartheta} =\bigcap_{k\in k_0+\NN_0} O_k,  
\end{equation*}
have complements $\RR_{\infty}^d \setminus O_k$ in $\RR_{\infty}^d$ \textit{without isolated points.\/} Then 
 every open set  $O_k\Subset O$ is regular for the Dirichlet problem.
It suffices to prove that the equality $\mu ( O_k\cap E)=0$ holds for every number  
 $k\in k_0+\NN_0$. By definition of polar sets, there is a  function $u\in \sbh_*(O)$ such that $u(E)=\{-\infty\}$. Consider functions 
\begin{equation*}
U_k=\begin{cases}
u \text{ \it  on $O\setminus O_k$},\\
\text{\it the harmonic extension of $u$ from $\partial O_k$ into $O_k$}\text{ on $O_k$},
\end{cases}   
\qquad k\in k_0+\NN_0.
\end{equation*}
We have  $U_k\in \sbh_*(O)$, and $U_k$ is bounded from below in $\supp {\vartheta} \Subset O_k$. Hence
\begin{multline*}\label{<U}
-\infty <\int_O U_k \dd {\vartheta}
\overset{\eqref{balnumu}}{\leq}
\int_O U_k \dd \mu=
\left(\int_{O\setminus (O_k\cap E)}+\int_{O_k\cap E}\right) U_k \dd \mu
\\
=\int_{O\setminus (O_k\cap E)} U_k \dd \mu+(-\infty)\cdot \mu(O_k\cap E)
\overset{\eqref{{infty}0}}{\leq} \mu(O) \sup_{\supp \mu} U_k+(-\infty)\cdot \mu(O_k\cap E).
\end{multline*}
Thus, we have  $\mu(O_k\cap E)=0$.
\end{proof}

Generally speaking, Proposition \ref{Pr_pol} is not true for $\har(O)$-balayage.  An implicit example built in \cite[Example]{MenKha19}. 
We will indicate in Example \ref{5} one more constructive and general way of building in this direction.

\begin{example}[{\rm development of one example  of T. Lyons \cite[XIB2]{Koosis}}] \label{5} Consider 
\begin{equation}\label{Eas}
O=\BB,  \quad 0<r_0<r<1,  \quad \vartheta\overset{\eqref{df:spb}}{:=}
\frac1{b_dr_0^d}\lambda_d\bigm|_{r_0\BB}, 
\quad \mu \overset{\eqref{df:spb}}{:=}
\frac1{b_dr^d}\lambda_d\bigm|_{r\BB}  
\end{equation} 
Easy to see that $\vartheta \preceq_{\sbh(\BB)} \mu$. Let $E=(e_j)_{j\in \NN}\Subset r\BB\setminus r_0\overline \BB$ be a polar countable set without limit point in $r\BB\setminus r_0\overline \BB$.  Surround each point $e_j\in E$ with a ball $B(e_j,r_j)$ of such a small radius  $r_j>0$ that the union of all these balls is contained in
$r\BB\setminus r_0\overline \BB$. Consider a measure 
\begin{align*}
\mu_E&\overset{\eqref{mB}}{:=}\mu-\frac{1}{b_dr^d}\sum_{j\in \NN} \lambda_d\bigm|_{B(e_j,r_j)}+\frac{1}{b_dr^d}\sum  \lambda_d(e_j,r_j)\delta_{e_j}\\
&\overset{\eqref{Eas}}{=}
\frac1{b_dr^d}\lambda_d\bigm|_{r\BB}-\frac{1}{b_dr^d}\sum_{j\in \NN} \lambda\bigm|_{B(e_j,r_j)}+\frac{1}{r^d}\sum_{j\in\NN} r_j^d \delta_{e_j} .
\end{align*}
By construction, the measure $\mu_E$ is $\har(\BB)$-balayage of measure $\vartheta$, 
but \begin{equation*}
\mu_E(E)=\frac{1}{r^d}\sum_{j\in \NN} r_j^d >0.
\end{equation*}
in direct contrast to Proposition \ref{Pr_pol}. 
\end{example}

\section{Potentials of  charges and measures}\label{Pot}
\setcounter{equation}{0}

Further everywhere we will assume for simplicity and brevity that 
\begin{equation}\label{OD}
(O\subset \RR^d)\Leftrightarrow(\infty \notin O), \quad (D\subset \RR^d)\Leftrightarrow(\infty \notin D)
\end{equation}
in addition to \eqref{ODj}. If  $\infty \in O$, $o\in \RR_{\infty}^d\setminus O$,   
we can always easily go to cases  \eqref{OD} using a  inversion ${\star_{o}}$, 
and the Kelvin transforms \eqref{stK}.

\begin{definition}[{\rm \cite{R}, \cite[Definition 2]{Kha03}, \cite[3.1, 3.2]{KhaRoz18}, \cite{Chi18}}]\label{df:pot} 
Let ${\mu} \in \Meas_{\comp}(\RR^d)$ be charge with compact support. Its {\it potential\/} 
 is the  function $\pt_{\mu}\in \dsbh_*(\RR^d)$
defined by      
\begin{equation}\label{{pmu}p} 
\pt_{\mu}(y)\overset{\eqref{{kK}K}}{:=}\int K_{d-2}(x,y) \dd \mu (x), 
\end{equation}
where the kernel  $K_{d-2}$ is defined in Definition \ref{df:kK}
by the function $k_q$ from \eqref{{kK}k}. The values of potential 
${\pt}_{\mu}(y)\in \RR_{\pm\infty}$ is well defined for 
all 
\begin{equation*}
\begin{split}
y\overset{\eqref{dom}}{\in} \Dom_{-\infty} {\pt}_{\mu}&=
\left\{y\in \RR^d\colon \int_{0}\frac{\mu^-(y,t)}{t^{m-1}} \dd t\overset{\eqref{mB}}{<}+ \infty 
\right\},
\\
y\overset{\eqref{dom}}{\in} \Dom_{+\infty} {\pt}_{\mu}&= 
\left\{y\in \RR^d\colon \int_{0}\frac{\mu^+(y,t)}{t^{m-1}} \dd t\overset{\eqref{mB}}{<}+ \infty 
\right\},
\\
y\overset{\eqref{dom}}{\in} \Dom_{\pm \infty} {\pt}_{\mu}&= \Dom_{-\infty} {\pt}_{\mu}\bigcup  \Dom_{+\infty} {\pt}_{\mu},
\\
y\overset{\eqref{dom}}{\in} \dom {\pt}_{\mu}&=\Dom_{-\infty} {\pt}_{\mu}\bigcap \Dom_{+\infty}{\pt}_{\mu},
\end{split}
\end{equation*}
and their complements $\RR^d\setminus \Dom_{-\infty} {\pt}_\mu$ and $\RR^d\setminus \Dom_{+\infty} {\pt}_{\mu}$ are  {\it polar sets\/} in $\RR^d$.
 \end{definition}

If $\mu\in \Meas_{\comp}^+(O)$ be a $H$-balayage of a measure ${\vartheta} \in \Meas_{\comp}^+(O) $, then we consider the potential
\begin{equation}\label{pmunu}
{\pt}_{\mu-{\vartheta}}\overset{\eqref{{pmu}p}}{:=}{\pt}_\mu-{\pt}_{\vartheta} \in \dsbh(\RR^d)
\end{equation}
where under the conditions $d>1$ and $1\in H$ it is natural to set ${\pt}_{\mu-{\vartheta}}(\infty) :=0$. 
The latter is based on the following
\begin{proposition}\label{prO} Let $\mu \in \Meas_{\comp}(\RR^d)$. Then
\begin{equation}\label{pnu0}
{\pt}_{\mu}(x)\overset{ \eqref{{kK}k}}{=}\mu (\RR^d)k_{d-2}\bigl(|x|\bigr)+O\bigl(1/|x|^{d-1}\bigr),  
\quad x\to \infty.
\end{equation}
\end{proposition}

\begin{proof} For $d=1$, we have
\begin{equation*}
\left|\pt_{\mu}(x)-\mu(\RR)|x|\right|\leq \int \bigl||x-y|-|x|\bigr|\dd |\mu|(y)\leq \int |y|\dd |\mu|(y)
=O(1), \quad |x|\to +\infty. 
\end{equation*} 
 
See \eqref{pnu0} for $d=2$ in \cite[Theorem 3.1.2]{R}. 

For $d>2$ and $|x|\geq 2\sup\bigl\{|y|\colon y\in \supp \mu\bigr\}$, we have 
\begin{multline*}
\left|{\pt}_{\mu}(x)-\mu (\RR^d)k_{d-2}\bigl(|x|\bigr)\right|
=\left|\int \left(\frac{1}{|x|^{d-2}}-\frac{1}{|x-y|^{d-2}}\right)
\dd \mu (y)\right|\\
\leq\int \left|\frac{1}{|x|^{d-2}}-\frac{1}{|x-y|^{d-2}}\right|
\dd |\mu|(y)
\leq \frac{2^{d-2}}{|x|^{2d-4}}\int \left||x-y|^{d-2}
-|x|^{d-2}\right| \dd |\mu|(y)
\\
\leq \frac{2^{d-2}}{|x|^{2d-4}}
\int |y||x|^{d-3} \sum_{k=0}^{d-3}\Bigl(\frac32\Bigr)^k 
\dd |\mu|(y)\leq 2\frac{3^{d-2}}{|x|^{d-1}}
\int |y|\dd |\mu|(y)=O\Bigl(\frac{1}{|x|^{d-1}}\Bigr).
\end{multline*}
\end{proof}

\begin{proposition}\label{pt_below} If 
\begin{equation}\label{muLo}
\mu \in \Meas_{\comp}^+(\RR^d),\quad
L\Subset \RR^d, \quad o\in \RR^d\setminus L,
\end{equation} 
then 
\begin{subequations}\label{pmul}
\begin{align}
\inf_{x\in L} {\pt}_{\mu}(x)&\overset{\eqref{{kK}k}}{\geq} \mu(\RR^d)k_{d-2}\bigl(\dist(L,\supp \mu)\bigr),
\tag{\ref{pmul}i}\label{{pmul}i}
\\
\inf_{x\in L} {\pt}_{\mu-\delta_o}(x)&\overset{\eqref{pmunu}}{\geq} \mu(\RR^d)k_{d-2}\bigl(\dist(L,\supp \mu)\bigr)-
k_{d-2}\left(\sup_{x\in L}|x-o|\right)
\tag{\ref{pmul}o}\label{{pmul}o}
\end{align}
\end{subequations}
\end{proposition}
\begin{proof} The case $d=1$ is trivial.  Consider the cases $d\geq 2$.
   If $\dist(L,\supp \mu)=0$, 
then the right-hand sides in the inequalities \eqref{pmul}
are equal to $-\infty$, and the inequalities \eqref{pmul} are true.  Otherwise, by Definition \ref{df:pot}, we obtain
\begin{multline}\label{est:pinf}
 {\pt}_{\mu}(x)=\int k_{d-2}\bigl(|x-y|\bigr)\dd \mu(y)\geq
\inf_{y\in \supp \mu} k_{d-2}\bigl(|x-y|\bigr)
\mu (\RR^d) \\
\geq \inf_{y\in \supp \mu} k_{d-2}\left(\inf_{y\in \supp \mu} |x-y|\right) \mu (\RR^d)=
\mu(\RR^d)k_{d-2}\bigl(\dist(x,\supp \mu)\bigr),
\end{multline}
since the function $k_q$ from \eqref{{kK}k} is \textit{increasing,\/}
which implies the inequality \eqref{{pmul}i}
 after applying the operation $\inf_{x\in L}$ to both sides of inequality \eqref{est:pinf}. Using \eqref{{pmul}i}, we have
\begin{multline*}
\inf_{x\in L} {\pt}_{\mu-\delta_o}(x) \overset{\eqref{{pmu}p}}{=}
\inf_{x\in L} \bigl({\pt}_{\mu}(x)-k_{d-2}
\bigl(|x-o|\bigr)\bigr)\geq 
\inf_{x\in L} {\pt}_{\mu}(x)- \sup_{x\in L}k_{d-2}
\bigl(|x-o|\bigr)
\\
\overset{\eqref{{pmul}i}}{\geq}\mu(\RR^d)k_{d-2}\bigl(\dist(L,\supp \mu)\bigr)-
k_{d-2}\left(\sup_{x\in L}|x-o|\right)
\end{multline*}
which gives the inequality \eqref{{pmul}o}. 

\end{proof}

\subsection{Duality Teorem for $\har(O)$-balayage}

\begin{dualtheorem}[{\rm for $\har(O)$-balayage}]\label{DT1}
If  a measure $\mu \in \Meas_{\comp}^+(O)$ is a\/ $\har (O)$-ba\-la\-ya\-ge of a measure ${\vartheta}\in \Meas_{\comp}^+ (O)$,  then 
\begin{subequations}\label{pmu0}
\begin{align}
{\pt}_{\mu}&\in \sbh_*(\RR^d)\cap 
\har(\RR^d\setminus \supp \mu), 
\tag{\ref{pmu0}p}\label{{pmu0}p}
\\ 
{\pt}_{\mu}&=  {\pt}_{\vartheta} \text{ on $\RR^d \setminus \inhull_O (\supp {\vartheta} \cup \supp \mu \cup)$}.
\tag{\ref{pmu0}=}\label{{pmu0}o}
\end{align}
\end{subequations}

\underline{Conversely}, suppose that there is a subset $S\Subset O$,
and  a function $p$ such that 
\begin{subequations}\label{p}
\begin{align}
p&\overset{\eqref{{pmu0}p}}{\in} \sbh (O)\cap  \har (O\setminus S),
\tag{\ref{p}p}\label{{pmu0+}p}
\\
p&\overset{\eqref{{pmu0}o}}{=} {\pt}_\vartheta \quad\text{on $O\setminus S$}
\quad\text{for a measure $\vartheta \in \Meas^+{\comp}(O)$.}
\tag{\ref{p}=}\label{{pmu0+}o}.
\end{align}
\end{subequations}
Then the  Riesz measure
\begin{equation}\label{mu}
\mu:=\varDelta_p \overset{\eqref{df:cm}}{:=}c_d\bigtriangleup p \overset{\eqref{{pmu0+}p}}{\in} \Meas^+(\clos S)\subset \Meas_{\comp}^+(O)
\end{equation} 
of this function $p$ is a $\har(O)$-balayage of the measure $\vartheta\in  \Meas^+{\comp}(O)$. 
\end{dualtheorem}
\begin{proof}  The first  property \eqref{{pmu0}p}  is  evidently. 
For each $y\in \RR^d$, the kernel $K_{d-2}(\cdot,  y)$
is harmonic on $\RR^d\setminus \{y\}$. By Proposition \ref{Pr:munuh},
for $h:=K_{d-2}(\cdot , y)$ in \eqref{bh}, we have  
\begin{equation}\label{p0}
{\pt}_{{\vartheta}}(y)=\int K_{d-2}(x, y) \dd {\vartheta} (x)
\overset{ \eqref{bh}}{=}\int K_{d-2}(x, y) \dd \mu (x)={\pt}_{\mu}(y)
\end{equation}
for all  $y\in \inhull_O (\supp \mu \cup \supp {\vartheta})$. This gives \eqref{{pmu0}o}.

In the \underline{opposite direction}, 
we can extend the function $p$ to $\RR^d$ so that 
$p={\pt}_{\vartheta}$ on $\RR^d\setminus S$. In view of \eqref{p}, 
 we have $p\in \sbh(\RR^d)\cap \har(\RR^d\setminus S)$, and, by Proposition \ref{prO},
\begin{equation}\label{asp}
p(x)-\vartheta (O)k_{d-2}(|x|)\overset{\eqref{pnu0}}{=}
p(x)-{\pt}_{\vartheta}(x)+O\bigl(1/ |x|^{d-1}\bigr) \overset{\eqref{{pmu0+}o}}{=} 
O\bigl(1/ |x|^{d-1}\bigr), \quad x\to  \infty. 
\end{equation}
Hence the function $p$ is a potential with the Riesz measure \eqref{mu}, 
and $\mu(O)=\vartheta (O)$, i.\,e., $p={\pt}_\mu$ (see \cite[Theorem 16]{Arsove53p} and \cite[6.7.2]{HaymanII} for $d=2$, and \cite[3.10]{HK} for $d>2$).  
Further, we can use the following
\begin{lemma}[{\rm \cite[Lemma 1.8]{Gardiner}}]\label{l1} Let $F$ be a compact subset of $\RR^d$, let $h\in \har (F)$, and $b \in \RR_*^+$. Then  there are points $y_1,y_2,\dots, y_m$ in $\RR^d\setminus F$ and constants  $a_1, a_2, \dots , a_m\in \RR$ such that
\begin{equation}\label{hk}
\Bigl|h(x)-\sum_{j=1}^{m} a_j k_{d-2}\bigl(|x-y_j|\bigr)\Bigr|<b
\quad\text{for all $x\in F$}.
\end{equation}
\end{lemma}
Applying Lemma \ref{l1} to the compact set 
$F\overset{\eqref{{pmu0+}p}}{:=}\clos S\cup \supp \vartheta \Subset O$
and a function $h\in \har (O)$, we obtain
\begin{multline*}
\Bigl|\int_{F} h \dd  (\mu -\vartheta)\Bigl|\overset{\eqref{{pmu0+}o}}{=}
\Bigl|\int_{F} h \dd (\mu -\vartheta)-\sum_{j=1}^{m}a_j \bigl({\pt}_{\mu} (y_j)- {\pt}_{\vartheta}(y_j)\bigr) \Bigl|
\\
\leq \sup_{x\in F} \Bigl| h (x)-\sum_{j=1}^{m} a_jk_{d-2}\bigl(|x-y_j|\bigr)\Bigr| \bigl(\mu(O)+\vartheta (O)\bigr)
\leq b 
\bigl(\mu(O)+\vartheta (O)\bigr)
\end{multline*}
for each  $b 
\in \RR_*^+$. Hence the measure $\mu$  is a $\har(O)$-balayage of $\vartheta$.   
\end{proof}

\begin{corollary}
Let  $\vartheta, \mu \in \Meas_{\comp}(O)$, $\supp  \vartheta \cup \supp  \mu \subset S \Subset O$. If $\mu$ is  a  balayage of $\vartheta$ for the class
\begin{equation}\label{Hk}
H=\bigl\{\pm k_{d-2} \bigl(|y-\cdot |\bigr)\colon y\in \RR^d \setminus \clos S \bigr\},
\end{equation}
then  $\mu$ is  a  $\har(O)$-balayage of $\vartheta$.
\end{corollary}
\begin{proof} We have \eqref{p0}  for all $y\in \RR^d\setminus \clos S$. 
 By Duality Theorem \ref{DT1}, $\vartheta \preceq_{\har(O)} \mu$.   
\end{proof}

\subsection{Arens\,--\,Singer measures and their potentials}\label{ASP} 

If $x\in O$ and $\vartheta:=\delta_x\preceq_{\har(O)}\mu\in \Meas^+_{\comp}(O)$, i.\,e., $\mu $ is a Arens\,--\,Singer measure for $x\in O$ from  Example \ref{sbhAS},  then potential 
\begin{equation}\label{ASp}
{\pt}_{\mu-\delta_x}(y)={\pt}_{\mu}(y)-K_{d-2}(x, y), \quad y\in \RR^d\setminus \{x\}, 
\end{equation}
satisfies conditions 
\begin{equation}\label{ASpc}
\begin{split}
{\pt}_{\mu-\delta_x}&\in \sbh(\RR^d_{\infty}),\quad  {\pt}_{\mu-\delta_x}(\infty):=0,
\\ 
 {\pt}_{\mu-\delta_x}&\equiv 0 \quad\text{on $\RR^d_{\infty} \setminus \inhull_O \bigl(\{x\} \cup \supp \mu \bigr)$}\\
{\pt}_{\mu-\delta_x}(y)&\leq -K_{d-2}(x,y)+O(1)\quad\text{as $x\neq y\to x$}. 
\end{split}
\end{equation}
Recall that the function $V\in \sbh_*\bigl(\RR^d_{\infty}\setminus \{x\}\bigr)$  
is called a {\it Arens\,--\,Singer potential  on $O$ with pole at\/} $x\in O$  \cite[3]{Gamelin}, \cite{Sarason}, \cite{Kha96}, \cite{Kha01II}, \cite{Kha03}, \cite[Definition 6]{Kha07},  \cite[\S~4]{KudKha09}, if this function $V$ satisfies conditions
\begin{equation}\label{ASpc+}
\begin{split}
V&\equiv 0 \quad\text{on $\RR^d_{\infty} \setminus S(V))$ for a subset $S(V)\Subset O$},\\
V(y)&\leq -K_{d-2}(x,y)+O(1)\quad\text{for $x\neq y\to x$}. 
\end{split}
\end{equation}
The class of all Arens\,--\,Singer potential  on $O$ with pole at $x\in O$ denote by $PAS_x(O)$. In this class $PAS_x(O)$ we will consider a special subclass 
\begin{equation}\label{PAS1}
PAS_x^1(O):=\bigl\{ V\in PAS_x(O)\colon V(y)
=-K_{d-2}(x,y) +O(1) \text{ for $x\neq y\to x$}\bigr\}
\end{equation}

By Duality Theorem \ref{DT1}, we have
\begin{dualtheoA}[{\rm \cite[Proposition 1.4, Duality Theorem]{Kha03}}] The mapping 
\begin{equation}\label{mcP}
\mathcal P_x \colon \mu \longmapsto {\pt}_{\mu-\delta_x}
\end{equation}
is the affine bijection from $AS_x(O)$ onto $PAS_x(O)$ with inverse mapping
\begin{equation}\label{P-1}
\mathcal P_x^{-1} \colon V
\overset{\eqref{df:cm}}{\longmapsto} 
c_d {\bigtriangleup}V \bigm|_{\RR^d\setminus \{x\}}+\left(1-\limsup_{x\neq y\to x} \frac{V(y)}{-K_{d-2}(x,y)}\right)\cdot \delta_x.
\end{equation}
Let $x\in \Int Q=Q\Subset O$. The restriction of $\mathcal P_x$ to the class 
\begin{equation}\label{interAi0}
\bigl\{\mu \in AS_x(O)\colon \supp \mu \cap Q=\varnothing \bigr\}
\end{equation}
define a bijection from class \eqref{interAi0}  onto class\/
{\rm (see \eqref{PAS1})}
\begin{equation}\label{PVl0}
PAS_x^1(O) \bigcap \har \bigl(Q\setminus \{x\}\bigr). 
\end{equation}
The restriction of $\mathcal P_x$ to the class 
\begin{equation}\label{interAi}
\bigl\{\mu \in AS_x(O)\colon \supp \mu \cap Q=\varnothing \bigr\}
\bigcap  \bigl(C^{\infty}(O) \dd \lambda_d \bigr)
\end{equation}
define also  a bijection from class \eqref{interAi}  onto class
\begin{equation}\label{PVl}
PAS_x^1(O) \bigcap \har \bigl(Q\setminus \{x\}\bigr) \bigcap C^{\infty} \bigl(O\setminus \{x\}\bigr). 
\end{equation}
\end{dualtheoA}
 This transition from the main bijection $\mathcal P_x$ to the bijection
from \eqref{interAi0} onto \eqref{PVl0} or 
 from \eqref{interAi} onto  \eqref{PVl} by restriction of $\mathcal P_x$  to \eqref{interAi0} or \eqref{interAi}  is quite obvious.

\subsection{A generalization of Poisson\,--\,Jensen formula}\label{PJ} 

\begin{theorem}[{\rm generalized Poisson\,--\,Jensen formula for $\har(O)$-balayage}]\label{PJf}
Let a measure $\mu\in \Meas_{\comp}^+(O)$ be a $\har(O)$-balayage of  a measure ${\vartheta} \in \Meas_{\comp}^+(O) $.  If $u\in \sbh (O)$ is a function with  the Riesz measure $\varDelta_u\overset{\eqref{df:cm}}{:=}c_d\bigtriangleup u \in \Meas^+(O)$, 
then 
\begin{equation}\label{fPJ}
\int u\dd \vartheta+\int_{K} {\pt}_{\mu} \dd \varDelta_u=
\int_{K} {\pt}_{\vartheta} \dd \varDelta_u +\int u\dd \mu, 
\quad K:=\inhull_{O}(\supp \vartheta \cup \supp \mu).
\end{equation}
In particular, if 
\begin{equation}\label{tfin}
\int u\dd \vartheta > -\infty, 
\end{equation} 
then \eqref{fPJ} can be written as
\begin{equation}\label{fPJ+}
\int u\dd \vartheta=\int u\dd \mu-\int_{K} {\pt}_{\mu-\vartheta} \dd \varDelta_u.
\end{equation}

\end{theorem}
\begin{proof} 
Consider first the case \eqref{tfin}. Choose an  open set $O'$ such that $K\Subset O'\Subset O$. By the Riesz decomposition theorem
$u={\pt}_{\nu'}+h$ on $O'$, where $\nu':=\varDelta_u\bigm|_{O'}$ and  $h\in \har (O' )$.
Integrating this representation with respect to $\dd \vartheta$ and $\dd \mu$, we obtain 
\begin{subequations}\label{utm}
\begin{align}
\int u \dd \mu&= \int {\pt}_{\nu'} \dd \mu +\int h \dd \mu,
\tag{\ref{utm}$\mu$}\label{{utm}m}
\\
\int u \dd \vartheta &= \int {\pt}_{\nu'} \dd \vartheta  +\int h \dd \vartheta, 
\tag{\ref{utm}$\vartheta$}\label{{utm}o}
\end{align}
\end{subequations}
where the three integrals in \eqref{{utm}o}  are finite, although in the equality \eqref{{utm}m} the first two integrals can take simultaneously the value of $-\infty$, but the last integral in \eqref{{utm}m} is finite.  Therefore, the difference \eqref{{utm}m}$-$\eqref{{utm}o} of these two equalities is well defined:
\begin{equation}\label{if}
\int u \dd \mu-\int u \dd \vartheta= \int {\pt}_{\nu'}\dd \mu-
 \int {\pt}_{\nu'} \dd \vartheta  +\int h \dd (\mu-\vartheta),
\end{equation}
 where the first and third integrals can simultaneously take the value of $-\infty$, and the remaining integrals are finite. By Proposition \ref{Pr:munuh}, the\textit{ last integral\/} in \eqref{if} \textit{vanishes.\/}  
Using Fubini's theorem, in view of the symmetry property of kernel  
in \eqref{{pmu}p},  we have
\begin{multline}\label{ct}
\int {\pt}_{\nu'}\dd \vartheta=\int \int 
K_{d-2}(y,x) \dd \nu' (y)\dd \vartheta (x)\\
=\int  \int K_{d-2}(x,y) \dd \vartheta (x) \dd \nu' (y)=\int_{O'}{\pt}_{\vartheta} \dd \varDelta_u.
\end{multline}
and the same way
\begin{multline}\label{chmu}
\int {\pt}_{\nu'}\dd \mu
=\int \int K_{d-2}(y,x) \dd \nu' (y)\dd \mu (x)\\
=\int  \int K_{d-2}(x,y) \dd \mu (x) \dd \nu' (y)=
\int_{O'}{\pt}_{\mu} \dd \varDelta_u
\end{multline}
even if the integral on the left side of equalities \eqref{chmu} takes the value  $-\infty$ because the integrand $K_{d-2}(\cdot, \cdot )$ is bounded from above on the compact set  $\clos O'\times \clos O'$ \cite[Theorem 3.5]{HK}. Hence  equality  \eqref{if} can be rewritten as
\begin{equation*}
\int u \dd \mu-\int u \dd \vartheta= \int_{O'}{\pt}_{\mu} \dd \varDelta_u-
 \int_{O'}{\pt}_{\vartheta} \dd \varDelta_u 
=\int_{K}{\pt}_{\mu} \dd \varDelta_u-
 \int_{K}{\pt}_{\vartheta} \dd \varDelta_u 
\end{equation*}
since ${\pt}_{\mu}={\pt}_{\vartheta}$ on $O'\setminus K$. This gives equality \eqref{fPJ}
 in the case \eqref{tfin}. 

If condition\eqref{tfin}  is not fulfilled, then from the representation \eqref{{utm}o} it follows that the integral on the left-hand side of \eqref{ct} also takes the value $-\infty$. 
The equalities \eqref{ct} is still true  \cite[Theorem 3.5]{HK}. Hence, the first integral on the right side of the
formula \eqref{fPJ} also takes the value $-\infty$ and this formula \eqref{fPJ} remains true. 
\end{proof}
\begin{remark}
If $\vartheta :=\delta_x$ and $\mu:= \omega_D(x, \cdot)$ for $x\in D\Subset O$, then 
the formula  \eqref{fPJ+} is the classical  Poisson\,--\,Jensen formula \cite[Theorem 5.27]{HK}
\begin{multline*}\label{clPJ}
u(x)=\int u\dd \delta_x\overset{\eqref{fPJ+}}{=}
\int_{\partial D} u\dd \omega_D(x, \cdot)-
\int_{\clos D} {\pt}_{\omega_D(x, \cdot)-\delta_x} \dd \varDelta_u
\\=\int_{\partial D} u\dd \omega_D(x, \cdot)-
\int_{\clos D} g_{D}(\cdot, x) \dd \varDelta_u, \quad x\in D,
\end{multline*}
since 
$\delta_x\preceq_{\sbh(O)} \omega_D(x, \cdot)$ according to Example \ref{sbhAS}, 
and  the equalities 
\begin{equation*}
{\pt}_{\omega_D(x, \cdot)-\delta_x}={\pt}_{\omega_D(x, \cdot)}-{\pt}_{\delta_x}\overset{\eqref{{pmu}p}}{=}{\pt}_{\omega_D(x, \cdot)}-K_{d-2}(x,\cdot)=g_D(\cdot , x), \quad x\in D,
\end{equation*}  
are well known \cite[Ch.~4,\S~1,2]{Landkoff}. 
\end{remark}

\subsection{Duality Theorem for $\sbh(O)$-balayage}

\begin{dualtheorem}[{\rm for $\sbh(O)$-balayage}]\label{DTsbh}
If  a measure $\mu \in \Meas_{\comp}^+(O)$ is a\/ $\sbh (O)$-bala\-ya\-ge of a measure ${\vartheta}\in \Meas_{\comp}^+ (O)$,  then we have \eqref{pmu0}, and 
\begin{equation}\label{pmu0+}
{\pt}_{\mu}\geq {\pt}_{\vartheta} \quad \text{on $\RR^d$.}
\end{equation}
\underline{Conversely}, suppose that there is a subset $S\Subset O$,
and  a function $p$ such that we have \eqref{p}, and $p\geq {\pt}_{\vartheta}$ on $\clos S$.
Then the  Riesz measure \eqref{mu} of  $p$ is a $\sbh(O)$-balayage of $\vartheta$.
\end{dualtheorem} 
\begin{proof} If $\vartheta \preceq_{\sbh(O)}\mu$, then 
$\vartheta \preceq_{\har(O)}\mu$ and we have properties \eqref{pmu0} by Duality Theorem \ref{DT1}. For each $y\in \RR^d$, the function $K_{d-2}(\cdot, y)$
is subharmonic on $\RR^d$ and  \eqref{pmu0+} follows from  Definitions \ref{df:1} and \ref{df:pot}.  Conversely, if  a function $p$ is such as in \eqref{p},  then, by Duality Theorem \ref{DT1}, this function is a potential ${\pt}_{\mu}=p$ with the Riesz measure  \eqref{mu}, this measure $\mu\in \Meas^+_{\comp}(O)$ is a $\har(O)$-balayage for $\vartheta$, and $K:=\inhull (\supp \vartheta\cup \supp \mu)\subset  \clos S$. 
 Let $u\in \sbh_*(O)$. It follows from ${\pt}_\mu\geq {\pt}_{\vartheta}$ on $K$ that 
$\int_K {\pt}_{\vartheta} \dd \varDelta_u  \leq \int_K {\pt}_{\mu} \dd \varDelta_u$. Hence, by  the generalized Poisson\,--\,Jensen formula \eqref{fPJ}  from Theorem \ref{PJf},  we obtain
 $\int u \dd \vartheta \leq \int u\dd \mu$. 
\end{proof}

\subsection{Jensen measures and their potentials}\label{JP} 

By Example \ref{sbhJ}, if we choose $x\in O$ and $\vartheta:=\delta_x\preceq_{\sbh(O)}\mu\in \Meas^+_{\comp}(O)$, i.\,e., $\mu $ is a Jensen  measure for $x\in O$, then potential 
${\pt}_{\mu-\delta_x}(y)={\pt}_{\mu}(y)-K_{d-2}(x, y)$, $y\in \RR^d\setminus \{x\}$ ,
satisfies conditions \eqref{ASpc} and ${\pt}_{\mu-\delta_x}\geq 0$ on $\RR^d_{\infty}\setminus \{x\}$.

Recall  that a \textit{positive}  function $V\in \sbh^+\bigl(\RR^d_{\infty}\setminus \{x\}\bigr)$  is called a {\it Jensen potential  on $O$ with pole at\/} $x\in O$   if this function $V$ satisfies conditions \eqref{ASpc+} \cite[3]{Gamelin},  \cite{Anderson}, \cite{Kha03}, 
\cite{MOS}, \cite[Definition 8]{Kha07}, \cite[IIIC]{Koosis96}, \cite{Kha12}, \cite{KhaTalKha15}, \cite{BaiTalKha16}.
The class of all Jensen potential  on $O$ with pole at $x\in O$ denote by $PJ_x(O)\subset AS_x(O)$. In this class $J_x(O)$ we will consider a special subclass 
\begin{equation}\label{PJ1}
PJ_x^1(O)\overset{\eqref{PAS1}}{:=}
PJ_x(O)\bigcap PAS_x^1(O)\subset PAS_x^1(O).
\end{equation} 
By Duality Theorem \ref{DTsbh}, we have
\begin{dualtheoB}[{\rm \cite[Proposition 1.4, Duality Theorem]{Kha03}}] The mapping 
$\mathcal P_x$ defined  in \eqref{mcP} is the aff\-i\-ne bijection from $J_x(O)$ onto $PJ_x(O)$ with inverse mapping \eqref{P-1}.  

Let $x\in \Int Q=Q\Subset O$. The restriction of $\mathcal P_x$ to the class\/ {\rm (cf. \eqref{interAi0})}
\begin{equation}\label{interJi}
\bigl\{\mu \in J_x(O)\colon \supp \mu \cap Q=\varnothing \bigr\}
\end{equation}
define a bijection from class \eqref{interJi}  onto class\/
{\rm (see \eqref{PJ1}, cf. \eqref{PVl0})}
\begin{equation}\label{PVJ}
PJ_x^1(O) \bigcap \har \bigl(Q\setminus \{x\}\bigr). 
\end{equation}
The restriction of $\mathcal P_x$ to the class\/ {\rm (cf. \eqref{interAi})} 
\begin{equation}\label{interJio}
\bigl\{\mu \in J_x(O)\colon \supp \mu \cap Q=\varnothing \bigr\}
\bigcap  \bigl(C^{\infty}(O) \dd \lambda_d \bigr)
\end{equation}
define a bijection from class \eqref{interJio}  onto class\/
{\rm (cf. \eqref{PVl})} 
\begin{equation}\label{PVJo}
PJ_x^1(O)\bigcap \har \bigl(Q\setminus \{x\}\bigr) \bigcap C^{\infty} \bigl(O\setminus \{x\}\bigr). 
\end{equation}
\end{dualtheoB}
 This transition from the main bijection $\mathcal P_x$ to the bijection from \eqref{interJi} onto  \eqref{PVJ} or from \eqref{interJio} onto  \eqref{PVJo} by restriction of $\mathcal P_x$  to \eqref{interJi} or to \eqref{interJio} is quite obvious.

\section{Affine balayage of measures}\label{afm} The following definition is a special case of the general concept of affine balayage \eqref{ba}.
\setcounter{equation}{0}

\begin{definition}[\cite{KhaRozKha19}]\label{def:ab} Let  $O\subset \RR^d$ be an open subset, and  $S_o\Subset D$.  Let $\mathcal V$ be a class of Borel-measurable functions on $O\setminus S_o$. We say that a  measure $\mu\in \Meas^+(O)$ is an {\it affine balayage\/}  of a measure $\upsilon \in \Meas^+(O)$ {\it  for the class\/} $\mathcal V$, or, briefly, $\mu$ is an affine  $\mathcal V$-balayage of $\upsilon$,  
	{\it outside\/} $S_o$ and write $\upsilon
	\curlyeqprec_{S_o,\mathcal V} \mu$ if there exists  a constant 
	$C\in \RR$ such that
	\begin{equation}\label{df:ab}
	\int_{O\setminus S_o}  v\dd \upsilon \leq 
	\int_{O\setminus S_o}  v\dd \mu+C\quad\text{\it for all $v\in \mathcal  V$.}
	\end{equation}
	provided that all integrals are well defined by values from $\RR_{\pm\infty}$.
\end{definition}
\begin{remark}\label{rem:4} If $C=0$ for all $v\in \mathcal V$ in \eqref{df:ab}, then we 
are dealing with the {\it usual\/} $\mathcal V$-balayage from Definition \ref{df:1}.
We study only case, when $O=D$ is domain in $\RR^d$. Transferring these results to the case of $\#\Conn_{\RR_{\infty}^d} O<\infty$ is trivial (see Sec. \ref{hullin}, \eqref{ODj}).

 The case of $\#\Conn_{\RR_{\infty}^d} O=\infty$ is somewhat more complicated and cumbersome in terms of formulations and is not considered here.
\end{remark}

\section{Test subharmonic functions}\label{tsf}
\setcounter{equation}{0}
In the role of class $\mathcal V$, we consider various classes of test subharmonic functions
on $D\setminus S_o$ and their countable completions up \eqref{Fuparrow}, where $\partial D$ is non-polar,  $o\in S_o\neq \varnothing$. 

First, we define  subclasses of $\sbh_*(D\setminus S_o)$ that \textit{ vanish near  the boundary $\partial D$:}
 \begin{subequations}\label{s0}
\begin{align}
\sbh_0(D\setminus S_o):=\Bigl\{ v\in \sbh (D\setminus S_o)
\colon \lim_{D\ni x'\to x} &v(x')=0 \text{ for all $x\in \partial D$}\Bigr\},
\tag{\ref{s0}$_{0}$}\label{{s0}0}\\
\intertext{\it finite near the boundary $\partial D$:}
\hspace{-14pt}\sbh_{00}(D\setminus S_o):=\bigl\{ v\overset{\eqref{{s0}0}}{\in} \sbh (D\setminus S_o)
\colon  \text{there is\;} S(v)\Subset D& \text{\;such that\;} v\equiv 0 \text{\;on $D\setminus S(v)$}\bigr\},
\tag{\ref{s0}$_{00}$}\label{{s0}00}
\\
\intertext{\it positive near the boundary $\partial D$:}
\hspace{-6pt}\sbh_{+}(D\setminus S_o):=\bigl\{ v\in \sbh (D\setminus S_o)
\colon  \text{there is\;} S(v)\Subset D& \text{\;such that\;} v\geq 0 \text{\;on $D\setminus S(v)$}\bigr\},
\tag{\ref{s0}$_{+}$}\label{{s0}+}\\
\sbh_{+0}(D\setminus S_o)\overset{\eqref{{s0}+}}{:=}
\sbh_0(D\setminus S_o)\bigcap \sbh_{+}&(D\setminus S_o)\overset{\eqref{{s0}00}}{\supset} 
\sbh_{00}(D\setminus S_o).
\tag{\ref{s0}$_{+0}$}\label{{s0}+0}
\end{align}
\end{subequations}
\begin{proposition}\label{pr:0} A function $v\overset{\eqref{{s0}+0}}{\in} \sbh_{+0}(D\setminus S_o)$ continues as subharmonic function on $\RR^d_{\infty}\setminus S_o$ by  rule
\begin{equation}\label{pr:0v}
v(x):=\begin{cases}
v(x)&\quad\text{at $x\in D\setminus S_o$},\\ 
0&\quad\text{at $x\in \RR^d_{\infty}\setminus  D$}
\end{cases}
\quad \in \sbh(\RR^d_{\infty}\setminus S_o).
\end{equation}
\end{proposition}
The proof is obvious.
 
 Next we assume that the interior  $\Int S_o$ is  non-empty, i.\,e., there exists a point
 \begin{equation}\label{S0set}
 o\in \Int S_o\subset S_o\Subset D\subset \RR^d.
 \end{equation}
 
 Given constants 
\begin{equation}\label{bbpmr}
-\infty <b_-< 0 < b_+<+\infty, \quad 0<3r<\dist(S_o, \partial D),
\end{equation}
we define  the following classes of \textit{test subharmonic functions}
with different restrictions from above or\,/\,and  below in \eqref{tf}:
\begin{subequations}\label{tf}
\begin{align}
&\sbh_{\dots}(D\setminus S_o; \leq b_+):=\bigl\{v \in \sbh_{\dots}(D\setminus S_o)\colon v\leq  b_+ \text{ on }\partial S_o\bigr\},
\tag{\ref{tf}b}\label{{tf}b}
\\
&\sbh_{\dots}^+(D\setminus S_o; \leq b_+)\overset{\eqref{+S}}{:=}
\sbh_{\dots}(D\setminus S_o; \leq b_+)\bigcap\sbh^+(D\setminus S_o) ,
\tag{\ref{tf}b$^+$}\label{{tf}b+}
\\
\intertext{under notation  \eqref{Scup} for the outer $r$-parallel set, by \eqref{{tf}b},}
&\hspace{-10pt}\sbh_{+0}(D\setminus S_o; r,b_-< b_+):=
 \bigl\{v\overset{\eqref{{s0}+0}}{\in} \sbh_{+0}(D\setminus S_o; \leq b_+)\colon  b_-\leq v \text{ on } S_o^{\cup (3r)}\setminus S_o\bigr\},
\tag{\ref{tf}b$_{\pm}$}\label{{tf}bpm0}\\
&\hspace{-8pt}\sbh_{00}(D\setminus S_o; r,b_-< b_+):=
 \bigl\{v\overset{\eqref{{s0}00}}{\in} \sbh_{00}(D\setminus S_o; \leq b_+)\colon  b_-\leq v \text{ on } S_o^{\cup (3r)}\setminus S_o\bigr\},
\tag{\ref{tf}b$_{\pm}^0$}\label{{tf}bpm}\\
\intertext{and under the designation \eqref{v0} for averaging over spheres,}
&\hspace{-20pt}\sbh_{+0}(D\setminus S_o; \circ r,b_-< b_+):=\bigl\{v\overset{\eqref{{tf}b}}{\in} \sbh_{+0}(D\setminus S_o; \leq b_+)\colon  b_-\leq v^{\circ r}
 \text{ on } S_o^{\cup (2r)}\setminus S_o^{\cup r}\bigr\}.
\tag{\ref{tf}b$^\circ$}\label{{tf}bc}
\end{align}
\end{subequations}
If $\sbh_{\dots}^{\dots}(D\setminus S_o;\dots)$ is a class of test subharmonic functions 
from \eqref{tf}, then 
\begin{equation}\label{vup}
\sbh_{\dots}^{\dots \uparrow}(D\setminus S_o;\dots)
\overset{\eqref{Fuparrow}}{:=}\bigl(\sbh_{\dots}^{\dots}(D\setminus S_o;\dots)\bigr)^{\uparrow}
\end{equation}
is the countable completion \eqref{Fuparrow} of class $\sbh_{\dots}^{\dots}(D\setminus S_o;\dots)$  up.

\begin{proposition}\label{pr:12} We have the following  inclusions:
\begin{equation}\label{incv}
\begin{array}{ccccccc}
&{\sbh_{00}^+(D\setminus S_o,\leq b_+)}& {\subset} 
&{\sbh_{00}(D\setminus S_o; r,b_-< b_+)}&\\
&\cap& &\cap&\\
&{\sbh_{0}^+(D\setminus S_o,\leq b_+)}&\subset &{\sbh_{+0}(D\setminus S_o;  r,b_-< b_+)}&\\
&\cap& &\cap&\\
&{\sbh_{0}^{+\uparrow}(D\setminus S_o,\leq b_+)}&\subset &{\sbh_{+0}^{\uparrow}(D\setminus S_o; \circ r,b_-< b_+)}&
\end{array} 
\end{equation}
All inclusions here, generally speaking, are strict.
\end{proposition}
\begin{proof} Inclusions immediately follow from Definitions \eqref{bbpmr}--\eqref{vup}.   Example \ref{5} shows that all ``horizontal'' inclusions are strict. The first line of  ``vertical'' inclusions is strict in an obvious way. 
The second line of ``vertical'' inclusions is strict in the case when there are irregular points on the boundary $\partial D$ of the domain $D$ \cite[Lemma 5.6]{HK},  
since the limit values of the Green's function $g_D$ at such points are not zero, even if they exist \cite[Theorem 5.19]{HK}.
\end{proof}

\begin{gluingtheorem}[{\rm for test subharmonic functions}]\label{glth5}
 Let $D$ be a domain together with  \eqref{S0set}, $b_{\pm}, r$ are constants satisfying \eqref{bbpmr}. Then there is a constant
\begin{equation}\label{gDVgv}
B:=2\frac{b_+-b_-}{\const^+_{o,S_o,r}}:=\const^+_{o,S_o,r,b_{\pm}}>0.
\end{equation}
such that for any function $v\in \sbh_{+0} (D\setminus S_o; \circ r,b_-< b_+)$ 
we can construct a subharmonic  function $V\in \sbh_*(\RR_{\infty}^d\setminus \{o\})$ with properties  
\begin{subequations}\label{VKv}
\begin{align}
0<V&\overset{\eqref{gDVh}}{\in} \har^+\bigl(S_o\setminus \{o\}\bigr) \quad\text{on $S_o\setminus \{o\}$},
\tag{\ref{VKv}h}\label{gDVhv}\\
V&\overset{\eqref{gDV=}}{=}v \quad \text{on $D\setminus S^{\cup (3r)}$}, 
\tag{\ref{VKv}=}\label{gDV=v}
\\
v(x)\leq V(x)&\overset{\eqref{gDVleq}}{\leq} b_++ Bg_D(x,o) \quad \text{for all  $x\in  S^{\cup (3r)}\setminus S_o$},
\tag{\ref{VKv}+}\label{gDVleqv}
\\
0< V(x)&\overset{\eqref{gDVleq+}}{\leq} Bg_D(x,o) \quad \text{for all  $x\in  S_o\setminus \{o\}$},
\tag{\ref{VKv}$^+_0$}\label{gDVleqv+}
\\	
V(x)&\overset{\eqref{gDVo}}{=}-BK_{d-2} (x,o)+O(1) \quad\text{when $o\neq x\to o$}
\tag{\ref{VKv}o}\label{gDVov},
\\
V&\equiv 0 \quad\text{on $\RR^d_{\infty}\setminus D$.}
\tag{\ref{VKv}$_0$}\label{gDV=0}
\end{align}
\end{subequations}
Besides, for any  $v\in  \sbh_{+0}^{\uparrow}(D\setminus S_o; \circ r,b_-< b_+)$ we get  a function $V\colon  \RR_{\infty}^d\setminus \{o\} \to \RR_{-\infty}$ 
as the limit of the increasing sequence of functions satisfying the conditions \eqref{gDVhv}--\eqref{gDVov} with the same properties  \eqref{gDVhv}--\eqref{gDVov}, but with a weaker property instead of \eqref{gDV=0}, more precisely 
\begin{equation}\label{vuparr}
V\equiv 0 \quad\text{on $\RR^d_{\infty}\setminus \clos D$, \quad   $V\geq 0$ on $\partial D$,}
\tag{\ref{VKv}$'_0$}
\end{equation}
and such function $V$ is not necessarily upper semi-continuous on 
$\clos D\setminus S_o$.

This also holds true for any  functions 
$v\in \sbh_0^+(D\setminus S_o; \leq b_+)$ or  $v\in \sbh_0^{+\uparrow}(D\setminus S_o; \leq b_+)$, resp., together with additional constraint of positivity $V\geq 0$ on $\RR_{\infty}^d\setminus \{o\}$.
\end{gluingtheorem}
\begin{proof}
By Proposition \ref{pr:0} we can consider the function $v\in \in \sbh_{+0} (D\setminus S_o; \circ r,b_-< b_+)$ as defined on $\RR^d_{\infty}\setminus S_o$ by \eqref{pr:0v}, i.\,e., $v\equiv 0$ on $\RR^d_{\infty}\setminus D$, and $v\in \sbh(\RR^d_{\infty} \setminus S_o)$. By Gluing Theorem \ref{gl:th_es}  with open set $\mathcal O:=\RR^d\setminus \{o\}$, with constants $M_v\overset{\eqref{avvM}}{:=}b^+$,
$m_v\overset{\eqref{avvm}}{:=}b_-$, $M_g\overset{\eqref{gDVg}}{:=}\const^+_{o,S_o,r}>0$,
and a constant  $B$ from \eqref{gDVgv},
we  construct  a function $V\in \sbh_0(\RR^d_{\infty} \setminus \{o\})$, $V(\infty):=0$,	
with properties \eqref{VK} that go into properties  \eqref{gDVhv}--\eqref{gDVov} together with identity \eqref{gDV=0}. 
Note that we use the principle of domination in \eqref{gDVleqv}--\eqref{gDVleqv+}
for Green's functions to replace $D_r$ with $D$ 
since a domain $D_r$ from \eqref{gDVleq} 	is a subdomain of $D$ for \eqref{bbpmr}.
\end{proof}

\begin{proposition}\label{pr:13} Let  $b_{\pm}, r$ are constants satisfying \eqref{bbpmr},
$v\in \sbh_{+0}^{\uparrow}(D\setminus S_o; \circ r,b_-< b_+)$ (respectively, $v\in \sbh_0^{+\uparrow}(D\setminus S_o; \leq b_+)$).  Then there are a constant 
\begin{equation}\label{BV}
B\overset{\eqref{gDVgv}}{=}\const_{o,S_o,r,b_{\pm}}^+,
\end{equation}
and  an increasing sequence of Arens\,--\,Singer (resp., Jensen) potentials $V_n\in PAS_o(D)$ (resp., $V_n\in PJ_o(D)$), $n\in n_0+\NN_0$, such that 
\begin{subequations}\label{VPAS}
\begin{align}
0<&V_n\in \har\bigl(\Int S_o\setminus \{o\}\bigr),
\tag{\ref{VPAS}h}\label{{VPAS}h}
\\
BV_n&\underset{n_0\leq n\to \infty}{\nearrow}V \quad\text{on $D\setminus \{o\}$},
\tag{\ref{VPAS}$\uparrow$}\label{{VPAS}up}
\\
\intertext{where $V\colon  \RR_{\infty}^d\setminus \{o\} \to \RR_{-\infty}$  is a function with properties \eqref{gDVhv}--\eqref{gDVov}, \eqref{vuparr},}
\lim_{o\neq x\to o}&\frac{V_n(x)}{-K_{d-2}(x,o)}=1,
\tag{\ref{VPAS}o}\label{{VPAS}upo}
\\
BV_n&\leq  b_++Bg_D(\cdot ,o)\quad\text{on $S_o^{\cup (3r)}\setminus \{o\}$}, \quad n\in n_0+\NN_0.
\tag{\ref{VPAS}b}\label{{VPAS}b}
\end{align}
\end{subequations}
\end{proposition}
\begin{proof} The classes $PAS_o(D)$ and $PJ_o(D)$ are closed relative to the maximum. By Proposition \ref{pr:up}, it suffices to prove Proposition \ref{pr:13} only for  functions  $v\in \sbh_{+0}(D\setminus S_o; \circ r,b_-< b_+)$ (resp., $v\in \sbh_0^{+}(D\setminus S_o; \leq b_+)$). For a function $v\in \sbh_{+0}(D\setminus S_o; \circ r,b_-< b_+)$, we  consider a function $V$ from Gluing Theorem 
\ref{glth5}. For each number $n\in \NN$ we put in correspondence an open set  
$ O_n:=\bigl\{x\in \RR_{\infty}^d\setminus \{o\}\colon V(x)<1/n\bigr\}\supset O_{n+1}$,
and  a function $v_n$ such that 
\begin{enumerate}[{i)}]
\item\label{iv} this function $v_n$ vanishes  on all connected components $\conn_{\RR_{\infty}^d}(O_n,x_j)\in \Conn_{\RR_{\infty}^d} O_n$ that met with complement  $\RR_{\infty}^d\setminus D$ of $D$, i.\,e.,  $v_n\equiv 0$ on every  connected  component $\conn_{\RR_{\infty}^d}(O_n,x_j)$ admitting a  representer $x_j\in \RR_{\infty}^d\setminus D$,
\item\label{iiv}  $v_n:=V-1/n$ on the rest of $\RR_{\infty}^d\setminus \{o\}$.
\end{enumerate}
By the construction \ref{iv})--\ref{iiv}), this functions $v_n$  are subharmonic on $\RR^d_{\infty}\setminus \{0\}$,
have a compact support in $D$, i.\,e.,  $v_n \overset{\eqref{{s0}00}}{\in}\sbh_{00} (D\setminus \{o\})$, and form an increasing sequence $v_n\underset{n\to \infty}{\nearrow} V$ on $\RR^d_{\infty}\setminus \{o\}$. In view of \eqref{gDVov}, there exists the limit
\begin{equation*}
\lim_{o\neq x\to o}\frac{v_n(x)}{-K_{d-2}(x,o)}=B.
\end{equation*}
Thus, if we set $V_n:=\frac{1}{B}\,v_n$, 
then $V_n\in PAS_o(D)$ (see definition \eqref{ASpc+} from Subsec. \ref{ASP}) and  properties \eqref{VPAS} are fulfilled. 

In the case $v\in \sbh_0^{+}(D\setminus S_o; \leq b_+)$, we consider functions $v_n^+$ after \ref{iv})--\ref{iiv}) instead of $v_n$ and obtain 
$V_n\in PJ_o(D)$ (see definition  from Subsec. \ref{JP}) with  properties \eqref{VPAS}. 
\end{proof}

\begin{remark}\label{rem:5} 
Numerous methods and examples of constructing various classes of test subharmonic positive functions are described in articles \cite{KhaAbdRoz18}, \cite{KhaTam17L}.  Test subharmonic  alternating-sign functions can be obtained from them in the development of Example \ref{5} and in the consideration of potentials for  measure from such examples.
\end{remark}

\section{Criteria for subharmonic and harmonic functions}\label{Crit}
\setcounter{equation}{0}

\begin{Criterium}[{\rm for subharmonic functions}]\label{crit1}
Let  $D\neq \varnothing$ be a domain in $\mathbb \RR^d$ with non-polar boundary $\partial D$, $M\in \sbh (D)\cap C(D)$ be a function with the Riesz measure $\mu_M \in \Meas^+(D)$, and $u\in \sbh_*(D)$ be a function  with the Riesz measure $\upsilon_u \in  \Meas^+(D)$. Then  the following three  statements are equivalent:

\begin{enumerate}[{\rm [{\bf s}1]}]
\item\label{{s}1} There is a subharmonic  function $h\in \text{\rm sbh}_* (D)$ such that
\begin{equation}\label{uhM}
u+h\leq M \quad\text{on $D$}.
\end{equation} 
\item\label{{s}2} For any $S_o$  satisfying 
\begin{equation}\label{S0seto}
\varnothing \neq \Int S_o\subset S_o\Subset D\subset \RR^d
 \end{equation}
and for any constant $b_+\in \mathbb R^+_*$,  the measure  $\mu_M $ is an affine balayage of the measure $\upsilon_u$ for class $\sbh_0^{+\uparrow}(D\setminus S_o; \leq b_+)$   outside  $S_o$  {\rm (see  \eqref{{tf}b} and \eqref{vup})}. 
\item\label{{s}3} There are  $S_o$  as in \eqref{S0seto} and $b_+\in \mathbb R^+_*$ such that   $\mu_M $ is an affine balayage of 
$\upsilon_u$  for  the   class $\sbh_{00}^+(D\setminus S_o;\leq b_+)\bigcap  C^{\infty}(D\setminus S_o)$ outside $S_o$ {\rm (see \eqref{{s0}00} and \eqref{{tf}b+}).}
\end{enumerate}
\end{Criterium}
\begin{remark}\label{rem:6} A special case of Criterium \ref{crit1} were considered in \cite[Theorem 1]{KhaKha19} for one complex variable.  Various forms of implication  $[{\bf s}\ref{{s}1} ]\Rightarrow[{\bf s}\ref{{s}2}]$ for narrower classes of test subharmonic functions were obtained in the works \cite{KhaRoz18}, \cite{KhaTam17A}, \cite{KhaTam17L}.
\end{remark}

``Subharmonic'' Criterium \ref{crit1}  has a similar ``harmonic'' counterpart. 

\begin{Criterium}[{\rm for harmonic functions}]\label{crit2}
 Let the conditions of Theorem\/ {\rm 1} are fulfilled.  Then  the following three 
 statements are equivalent:

\begin{enumerate}[{\rm [{\bf h}1]}]
\item\label{{h}1} There exists a  harmonic  function $h\in \text{\rm har} (D)$ such that
 $u+h\leq M$ on $D$ as in \eqref{uhM}.

\item\label{{h}2+} For any  connected set $S_o$  from \eqref{S0seto} and  for any constants $r, b_{\pm}$ from\/ \eqref{bbpmr}, i.\,e., 
\begin{equation}\label{bbpmr+}
0<3r<\dist(S_o, \partial D), \quad  -\infty <b_-< 0 < b_+<+\infty,  
\end{equation}
there is a constant\/ $C\in\RR$ such that
{\rm (see \eqref{{tf}bc} and \eqref{vup})}
\begin{equation}\label{almB}
\int_{D\setminus S_o}v \dd \upsilon_u\leq 
\int_{D\setminus S_o^{\cup(3r)}}v \dd \mu_M+C \quad
\text{for all\quad $v\in \sbh_{+0}^{\uparrow}(D\setminus S_o;\circ r, b_-< b_+)$}.
\end{equation}
\item\label{{h}2} For any  connected set $S_o$  from \eqref{S0seto} and  for any constants
from\/ \eqref{bbpmr+},   $\mu_M $ is an affine balayage of  $\upsilon_u$  for  the  class $\sbh_{+0}^{\uparrow}(D\setminus S_o;  r, b_-< b_+)$  outside  $S_o$ {\rm (see \eqref{{tf}bpm0}, \eqref{vup})}.

\item\label{{h}3} There are  connected set $S_o$ as in \eqref{S0seto} and constants as in\/  \eqref{bbpmr+}  such that   $\mu_M $ is an affine balayage of 
 $\upsilon_u$  for $\sbh_{00}(D\setminus S_o; r,b_-< b_+)\bigcap  C^{\infty}(D\setminus S_o)$ outside $S_o$  {\rm (see \eqref{{tf}bpm})}.
\end{enumerate}
\end{Criterium}

\begin{remark}\label{rem:7} A very special case of Criterium  \ref{crit2} announced without proof in \cite[Theorem 2]{MenKha19} for functions of one complex variable and subharmonic functions  $u$ of the form $u=\ln |f|$, where $f$ is a non-zero holomorphic function on a  domain $D$ in the complex plane $\CC$. 
\end{remark}

\begin{remark}\label{rem:10} By the  inclusions \eqref{incv} of Proposition \ref{pr:12}, 
the implications {\bf s}\ref{{s}2}$\Rightarrow${\bf s}\ref{{s}3} from Criterium \ref{crit1}
and {\bf h}\ref{{h}2}$\Rightarrow${\bf h}\ref{{h}3} from Criterium \ref{crit2} are obvious.
\end{remark}

For $v\in L^1\bigl(B(x,r)\bigr)$, we define the averaging value of $v$ at the point $x$ on the ball 
$B(x,r)$ as (cf. with \eqref{v0})
\begin{equation}\label{v0B}
v^{\bullet r}(x)\overset{\eqref{df:spb}}{:=}\frac{1}{b_d}\int_{\BB} v (x+rs) \dd \lambda_d(s).
\end{equation}

\section{Proofs of implications {\bf s}\ref{{s}1}$\Rightarrow${\bf s}\ref{{s}2}
and {\bf h}\ref{{h}1}$\Rightarrow${\bf h}\ref{{h}2+}$\Rightarrow${\bf h}\ref{{h}2} from Criteria \ref{crit1} and \ref{crit2}}\label{proofs12}
\setcounter{equation}{0}

In view of \eqref{S0set},  there are  a point $o\in \Int S_o$ and a number  $r_o \in \RR_*^+$ such that 
\begin{equation}\label{oB}
o\in B(o,r_o)\overset{\eqref{{B}B}}{\subset} \overline{B}(o,r_o)\Subset \Int S_o\subset S_o\Subset D\subset \RR^d, \quad \text{and $u(o)\in \RR$, $M(o)\in \RR$}.
\end{equation}
Let $\mu \in J_o(D)$ be a Jensen measure for the point $o$ in the case {\bf s}\ref{{s}1} 
with the restriction
\begin{equation}\label{Bcap}
\overline B(o,r_o) \cap \supp \mu =\varnothing. 
\end{equation}
or  $\mu\in AS_o(D)$ be a Arens\,--\,Singer measure for the point $o$ in the case {\bf h}\ref{{h}1},  respectively.
It is follows from \eqref{uhM} in the cases  {\bf s}\ref{{s}1} or {\bf h}\ref{{h}1} that 
\begin{equation}\label{uhleqM}
\int_D u \dd \mu +\int_D h \dd \mu \leq \int_D M \dd \mu,  
\end{equation} 
where 
\begin{equation}\label{hh1}
\RR\ni h^{\bullet r_o}(o)\overset{\eqref{v0B}}{=}h(o)=\int_D h \dd \mu  \quad \text{in the case {\bf h}\ref{{h}1} with 
$h\in \har(D)$ and $\mu \in AS_o(D)$}.
\end{equation}
In the case {\bf s}\ref{{s}1}, we see by Proposition   \ref{pr:vstm} with 
$$
D:=O,  \quad \varsigma:=\delta_o, \quad \vartheta =\frac{1}{b_dr_o^d}\lambda_d\bigm|_{B(o,r_o)}\in \Meas_{\comp}^{1+}(D),
$$
  in view of \eqref{Bcap}, that   
\begin{equation}\label{muball}
\lambda_d^{r_o}:=\frac{1}{b_dr_o^d}\lambda_d\bigm|_{B(o,r_o)}\preceq_{\sbh(D)} \mu ,
\end{equation}
whence   
\begin{equation}\label{hbo}
\RR \ni h^{\bullet r_o} (o)\overset{\eqref{v0B}}{=} \int h\dd \lambda_d^{r_o}
\overset{\eqref{muball}}{\leq} \int h \dd \mu \text{ for
 $h\in \sbh_*(D)$, $\mu \in J_o(D)$ with \eqref{Bcap}}. 
\end{equation}
Thus, from \eqref{uhleqM} we obtain  by \eqref{hbo} or \eqref{hh1}   that,
both in the case of {\bf s}\ref{{s}1} and in the case of {\bf h}\ref{{h}1}, for a constant
$C_o:=-h^{\bullet r_o} (o)=\const_{h,o,r_o}\in \RR$, the following inequality is fulfilled:  
\begin{equation}\label{esthbull}
\int_D u \dd \mu \leq \int_D M \dd \mu +C_o \quad \text{for all }\mu \in
\left[
  \begin{array}{ccc}
     J_o(D)  & \text{in the case {\bf s}\ref{{s}1}}& \text{with \eqref{Bcap}}, \\
     AS_o(D) & \text{in the case {\bf h}\ref{{h}1}}.& {}\\
  \end{array}
\right.
\end{equation}
By  generalized Poisson\,--\,Jensen formula for $\har(O)$-balayage from Theorem  \ref{PJf}
in the form \eqref{tfin}--\eqref{fPJ+} for the functions $u,M\in \sbh_*(D)$ with $\vartheta :=\delta_o$, we have 
\begin{subequations}\label{PJuM}
\begin{align}
\int_D u \dd \mu &=u(o)+\int_D {\pt}_{\mu-\delta_o} \dd \upsilon_u, \quad u(o)\overset{\eqref{oB}}{\in} \RR,  
\tag{\ref{PJuM}u}\label{{PJuM}u}
\\
 \int_D M \dd \mu&=M(o)+\int_D {\pt}_{\mu-\delta_o} \dd \mu_M, \quad M(o)\in \RR.
\tag{\ref{PJuM}M}\label{{PJuM}M}
\end{align}
\end{subequations}
By \eqref{esthbull}--\eqref{PJuM}, we obtain
\begin{equation}\label{estV}
\begin{split}
\int_D {\pt}_{\mu-\delta_o} \dd \upsilon_u&\leq 
\int_D {\pt}_{\mu-\delta_o} \dd \mu_M+C_1,\\
\text{where } C_1&:=C_o+M(o)-u(o)=\const_{h,o,r_o,S_o,u,M},
\end{split}
\end{equation} 
for all $\mu\in J_o(D)$ with \eqref{Bcap} in the case {\bf s}\ref{{s}1} or for all $\mu\in AS_o(D)$  in the case {\bf h}\ref{{h}1} respectively.

In the case {\bf h}\ref{{h}1}, by Duality Theorem A with $x:=o$, if measures $\mu$ cover all class $AS_o(D)$, then the potentials ${\pt}_{\mu-\delta_o}$ run through the entire class $PAS_o(D)$ of Arens\,--\,Singer potentials. In the case {\bf s}\ref{{s}1}, by Duality Theorem B with $x:=o, O:=D,  Q:=B(o,r_o)$ in the version \eqref{interJi}--\eqref{PVJ}, if measures $\mu$ cover all class $J_o(D)$ with \eqref{Bcap}, then the potentials ${\pt}_{\mu-\delta_o}$ run through the entire class 
(see \eqref{PJ1})
\begin{equation}\label{PVJC}
PJ_o^1(D)\bigcap \har \bigl(B(o,r_o)\setminus \{o\}\bigr)\subset PJ_o(D)
\end{equation}
of Jensen potentials.  Thus, it is follows from \eqref{estV}  that  
\begin{equation}\label{estVV}
\int_D V \dd \upsilon_u\leq 
\int_D V \dd \mu_M+C_1, \quad \text{where the constant $C_1\in \RR$ independent of $V$},
\end{equation}
for all Arens\,--\,Singer potentials $V\in PAS_o(D)$  in the case {\bf h}\ref{{h}1}, and 
for all Jensen potentials  $V$ from the class \eqref{PVJC} in the case {\bf s}\ref{{s}1}.

Let $v\in \sbh_0^{+\uparrow}(D\setminus S_o; \leq b_+)$ in the case {\bf s}\ref{{s}1}
or $v\in \sbh_{+0}^{\uparrow}(D\setminus S_o; \circ r, b_-< b_+)$ in the case {\bf h}\ref{{h}1}, respectively. By Proposition \ref{pr:13},  there are  constants
from \eqref{BV} and  an increasing sequence of  Jensen (resp., Arens\,--\,Singer) potentials $V_n\in PJ_o(D)$ (resp., $V_n\in PAS_o(D)$), $n\in \NN$, satisfying \eqref{VPAS}.
For such potentials,  relation \eqref{estVV} entails the relations
\begin{equation}\label{estVVn}
\int_D BV_n \dd \upsilon_u\leq 
\int_D BV_n \dd \mu_M+BC_1 \leq 
\int_D V \dd \mu_M+BC_1, \quad V:=\lim_{n\to \infty} BV_n \quad\text{on $D\setminus \{o\}$},
\end{equation}
where the function\footnote{This function $V$ has nothing to do with potentials $V$ 
from \eqref{estVV}.} $V$  has all the properties  \eqref{gDVhv}--\eqref{gDVov}, \eqref{vuparr}, and  the  constant $BC_1\in \RR$ independent of $V_n$, $n\in \NN$.
We will present the integral  on the right-hand side  of inequalities \eqref{estVVn} in the form of sum of the integrals:
\begin{multline}\label{VmuM}
\int_D V \dd \mu_M=\left(\int_{D\setminus S_o^{\cup(3r)}} +
\int_{S_o^{\cup (3r)}\setminus S_o}+
\int_{S_o\setminus \{o\}} \right)V \dd \mu_M\\
\overset{\eqref{gDV=v}\text{-}\eqref{gDVleqv+}}{\leq} \int_{D\setminus S_o^{\cup(3r)}}v \dd \mu_M +
b_+\mu_M(S_o^{\cup (3r)}\setminus S_o)+
B\int_{S_o^{\cup (3r)}\setminus \{o\}} g_D(x,o)\dd \mu_M\\
\leq \int_{D\setminus S_o^{\cup(3r)}}v \dd \mu_M
+C_2,
\quad \text{where $C_2=\const_{o,S_o,r,b_{\pm},B,u,M}^+$}
\end{multline}
is a constant independent of $v$. In addition, in the case 
$v\in \sbh_0^{+\uparrow}(D\setminus S_o; \leq b_+)$, the function  $v$ is positive on $D\setminus S_0$, and we have 
\begin{equation}\label{+M}
\int_D V \dd \mu_M\overset{\eqref{VmuM}}{\leq} \int_{D\setminus S_o}v \dd \mu_M
+C_2\quad\text{in the case {\bf s}\ref{{s}1}}.
\end{equation}
Besides, in the case $v\in \sbh_{+0}^{\uparrow}(D\setminus S_o;  r, b_-< b_+)\subset \sbh_{+0}^{\uparrow}(D\setminus S_o;  \circ r, b_-< b_+)$,  the function $v$ is bounded from  below by $b_-\in \RR\setminus \RR^+_*$, and it is follows from \eqref{VmuM} that 
\begin{equation}\label{+Mb}
\begin{split}
\int_D V \dd \mu_M\overset{\eqref{VmuM}}{\leq} \int_{D\setminus S_o}v \dd \mu_M
+C_2-\mu_M(S_o^{\cup (3r)}\setminus S_o)b_-
=\int_{D\setminus S_o}v \dd \mu_M
+C_3,\\ \text{where $C_3=\const_{o,S_o,r,b_{\pm},B,u,M}^+$
in the case $v\in \sbh_{+0}^{\uparrow}(D\setminus S_o;  r, b_-< b_+)$}.
\end{split}
\end{equation} 

If the integrals in the right parts of \eqref{VmuM}, \eqref{+M}, \eqref{+Mb} are equal to $+\infty$, then there is nothing to prove. Otherwise, by the 
Beppo Levi's monotone convergence theorem for Lebesgue integral, \eqref{estVVn} together with \eqref{VmuM}--\eqref{+Mb} implies 
\begin{equation}\label{Levlim}
\int_D V \dd \upsilon_u\leq 
\int_{D\setminus S_o^*} v \dd \mu_M+C,\quad\text{where $C=\const_{o,S_o,r,b_{\pm},B,u,M}^+$}, 
\end{equation}
and $S_o^*:=S_o^{\cup (3r)}$ in the proof of statement {\bf h}\ref{{h}2+}, but $S_o^*:=S_o$ in the derivation of statement {\bf h}\ref{{h}2} from {\bf h}\ref{{h}2+}, as well as in the proof of statement {\bf s}\ref{{s}2}. According to 
\eqref{gDVhv}--\eqref{gDVleqv}, we obtain from  \eqref{Levlim}  that
\begin{multline*}
\int_{D\setminus S_o}v\dd \upsilon_u
\overset{\eqref{gDVhv}}{\leq}
\int_{S_0} V\dd  \upsilon_u
+\int_{D\setminus S_0} v\dd  \upsilon_u\\
\overset{\eqref{gDVleqv}}{\leq}
\int_{S_0} V\dd  \upsilon_u
+\int_{S_o^{\cup (3r)}\setminus S_0} V\dd  \upsilon_u
+\int_{D\setminus S_o^{\cup (3r)}} v\dd  \upsilon_u
\overset{\eqref{gDV=v}}{=}
\int_{D} V\dd  \upsilon_u
\overset{\eqref{Levlim}}{\leq}
\int_{D\setminus S_o^*} v \dd \mu_M+C,
\end{multline*}
where the constant $C$ is independent of $v$, and a set $S_o^*$ is defined immediately after \eqref{Levlim}.
\begin{remark}\label{rembD} 
We do not require any properties for the boundary $\partial D$ in the proofs of implications {\bf s}\ref{{s}1}$\Rightarrow${\bf s}\ref{{s}2}$\Rightarrow${\bf s}\ref{{s}3} and 
{\bf h}\ref{{h}1}$\Rightarrow${\bf h}\ref{{h}2+}$\Rightarrow${\bf h}\ref{{h}2}$\Rightarrow${\bf h}\ref{{h}3}. Besides, we do not use also the continuity of the majorizing function $M\in \sbh_*(D)$ from \eqref{uhM}.
Therefore these implications are true for arbitrary domain $D\subset \RR^d$ and for arbitrary majorizing function $M\in \sbh_*(D)$ in \eqref{uhM}.
\end{remark}

\section{Proofs of implications\/ {\rm {\bf s}\ref{{s}3}$\Rightarrow${\bf s}\ref{{s}1}}
and\/ {\bf h}{\rm \ref{{h}3}}$\Rightarrow${\bf h}{\rm\ref{{h}1}}
 from Criteria \ref{crit1} and \ref{crit2}}\label{proofs31}
\setcounter{equation}{0}

The basis of the proof of implications\/ {\rm {\bf s}\ref{{s}3}$\Rightarrow${\bf s}\ref{{s}1}}
and\/ {\bf h}{\rm \ref{{h}3}}$\Rightarrow${\bf h}{\rm\ref{{h}1}} is the following 
\begin{theoC}[{\rm \cite[Ch.~2, {\bf 8.2}, Corollary 8.1, II, 1, (i)-(ii)]{KhaRozKha19}}] Let $D\subset \RR^d$ be a domain, $H$ be a convex cone in $\sbh_*(D)$  containing constants, and
\begin{equation}\label{Minfty}
\mathcal M^{\infty}(D\setminus U_0)\overset{\eqref{MLl}}{:=}\Meas^+_{\comp}(D\setminus U_0)\bigcap \bigl(C^{\infty}(D)\dd \lambda_d\bigr).
\end{equation} 
Suppose that one of the following two conditions is fulfilled:
\begin{enumerate}[{\rm (a)}]
\item\label{sla} 
for any locally bounded from above sequence of functions  $( h_k)_{k\in \NN} \subset H$, the upper semi-continuous regularization of the upper limit 
\begin{equation*}
\limsup_{k\to\infty} h_k \quad\text{belong to $H$ provided that}\quad  \limsup_{k\to\infty} h_k(x)\not\equiv -\infty
\text{ on  $D$;}
\end{equation*}
\item\label{slb} 
 $H$ is sequentially closed in $L^1_{\loc}(D)$.
\end{enumerate}
 Let $u\in \sbh_*(D)$, $M\in C(D)$, and 
$0\neq \vartheta \in \Meas^+_{\comp}(D)$, $\supp \vartheta \subset U_0\Subset D$, where  $U_0$ is the domain.  If there is a constant $C\in \RR$ such that
\begin{equation}\label{affbal}
\int_D u\dd \mu \leq \int_D M\dd \mu+C
\quad\text{for all measures \; $\mu\overset{\eqref{Minfty}}{\in} \mathcal M^{\infty}(D\setminus U_0)$
such that $\vartheta\preceq_H \mu$},        
\end{equation} 
then there is a function $h\in H$ such that 
$u+h\leq M$ on $D$.
\end{theoC}

\subsection{From smooth Arens-Singer or Jensen measures to their potentials in Theorem C}

\begin{corollary}\label{JAS} 
Let  $U_o\Subset D\subset \RR^d$ are domains, $o\in U_o$, 
$u\in \sbh(D)$ with value $u(o)\neq -\infty$ and $M\in C(D)\cap \sbh (D)$ are a functions with the Riesz measures $\upsilon_u$ and $\mu_M$, respectively.

{\bf [h]} If $\mu_M$ is an affine balayage of $\upsilon_u$ outside $\{o\}$ for the class {\rm (see \eqref{PAS1}, \eqref{PVl})}
\begin{equation}\label{PVlo}
\mathcal V_{AS}^{\infty}(D\setminus U_o):=PAS_o^1(D) \bigcap \har \bigl(U_o\setminus \{o\}\bigr) \bigcap C^{\infty} \bigl(D\setminus \{o\}\bigr), 
\end{equation}
then there is a function $h\in \har(D)$ such that $u+h\leq M$ on $D$.  

{\bf [s]} If $\mu_M$ is an affine balayage of $\upsilon_u$ outside $\{o\}$ for the class {\rm (see \eqref{PJ1}, \eqref{PVJo})}
\begin{equation}\label{PVloU}
\mathcal V_J^{\infty}(D\setminus U_o):=PJ_o^1(D)\bigcap \har \bigl(U_o\setminus \{o\}\bigr) \bigcap C^{\infty} \bigl(D\setminus \{o\}\bigr)\subset \mathcal V_{AS}^{\infty}(D\setminus U_o), 
\end{equation}
then there is a function $h\in \sbh_*(D)$ 
such that $u+h\leq M$ on $D$.  
\end{corollary}
\begin{proof} We choose the convex cone 
\begin{equation}\label{Hhs}
H:=\left[ \begin{array}{ccc}
 \har(D)& \text{in the case {\bf [h]}}&\text{satisfying the condition \eqref{slb} from Theorem C},  \\ 
  \sbh_*(D)& \text{in the case {\bf [s]}}&\text{satisfying the condition \eqref{sla} from Theorem C},
\end{array} \right.
\end{equation}
respectively. By the generalized Poisson\,--\,Jensen formula from  Theorem \ref{PJf},  for all  potentials  $V\in \mathcal V_{AS}^{\infty}(D\setminus U_o)\supset \mathcal V_{AS}^{\infty}(D\setminus U_o)$ from \eqref{PVlo}--\eqref{PVloU} and $\vartheta:=\delta_o$,
we have 
\begin{equation}\label{uMV}
\begin{split}
\int_{D\setminus \{o\}} V \dd \upsilon_u&=u(o)+\int_D u (c_d \bigtriangleup V) \dd \lambda_d
\quad\text{for all $V\in \mathcal V_{AS}^{\infty}(D\setminus U_o)$},
\\ 
\int_{D\setminus \{o\}} V \dd \mu_M&=M(o)+\int_D M (c_d \bigtriangleup V) \dd \lambda_d 
\quad\text{for all $V\in \mathcal V_{AS}^{\infty}(D\setminus U_o)$}.
\end{split}
\end{equation}
By Duality Teorems A, the part \eqref{interAi}--\eqref{PVl}, or by Theorem B, the part \eqref{interJio}--\eqref{PVJo}, if functions $V$ run the class \eqref{PVlo} in the case {\bf [h]}
or the class \eqref{PVloU} in the case {\bf [s]}, respectively, the measures $\mu_V$ with densities 
$\dd \mu_V:=(c_d \bigtriangleup V) \dd \lambda_d$ run the classes of Arens\,--\,Singer measures from the class
(see \eqref{interAi})
\begin{equation}\label{interAio}
\mathcal M_{AS}^{\infty}  (D\setminus U_o):=\bigl\{\mu \in AS_o(D)\colon \supp \mu \cap U_o=\varnothing \bigr\}
\bigcap  \bigl(C^{\infty}(D) \dd \lambda_d \bigr)
\end{equation}
or the classes of Jensen measures from the class 
(see \eqref{interJio})
\begin{equation}\label{interJioo}
\mathcal M_J^{\infty} (D\setminus U_o):=\bigl\{\mu \in J_o(D)\colon \supp \mu \cap U_o=\varnothing \bigr\}
\bigcap  \bigl(C^{\infty}(D) \dd \lambda_d \bigr),
\end{equation} 
respectively. For the cases {\bf [h]} or {\bf [s]}, by Definition \ref{def:ab} the condition 
$\upsilon_u\curlyeqprec_H \mu_M$ with $H$ from \eqref{Hhs}
means that there is a constant $C\in \RR$ such that
\begin{equation*}
\int_{D\setminus \{o\}}  V\dd \upsilon_u 
\overset{\eqref{df:ab}}{\leq} 
\int_{D\setminus \{o\}}  V\dd \mu_M+C\quad
\text{for all $V$ from \eqref{PVlo} or \eqref{PVloU}, respectively.}
\end{equation*}
Hence, using \eqref{uMV}, we obtain with the constant 
$C_o:=C-u(o)+M(o)$ the inequalities
\begin{equation}\label{affbalo}
\int_D u\dd \mu \leq \int_D M\dd \mu+C_o 
\quad\text{for all $\mu\in \left[\begin{array}{cc}
 \mathcal M_{AS}^{\infty}  (D\setminus U_o)&
\text{from \eqref{interAio}}\\ 
\mathcal M_{J}^{\infty}  (D\setminus U_o)&
\text{from \eqref{interJioo}}
\end{array}\right.$, respectively,}        
\end{equation}
which corresponds to \eqref{affbal}. By Theorem C, this scale \eqref{affbalo} of inequalities proves the statements {\bf [h]} and {\bf [s]} for convex cone $H$ from \eqref{Hhs}, respectively. 
\end{proof}

\subsection{An embedding of smooth potentials into a class of smooth test functions}\label{ASJem}

Throughout this Subsec.\ref{ASJem}, $D\subset \RR^d$ is a domain with non-polar boundary $\partial D\subset \RR_{\infty}^d$, $o\in D$.

\begin{proposition}[{\rm a variant of  Phr\'agmen\,--\,Lindel\"of principle}]\label{FLp} If a function 
$v\in \sbh (D\setminus \{o\})$
satisfies the conditions
\begin{equation}\label{limgv}
\limsup_{o\neq x\to o}\frac{v(x)}{-K_{d-2}(x,o)} \leq 0, 
\quad \limsup_{D\ni x\to \partial D}\bigl( v(x)\bigr) \leq 0,
\end{equation} 
then $v\leq 0$ on $D\setminus \{o\}$. In particular,
if a function $V\in \sbh (D\setminus \{o\})$
satisfies the conditions
\begin{equation}\label{limgV}
\limsup_{o\neq x\to o}\frac{V(x)}{-K_{d-2}(x,o)} \leq c\in \RR^+, 
\quad \limsup_{D\ni x\to \partial D}\bigl( V(x)\bigr) \leq 0,
\end{equation}
then $V\leq cg_D(\cdot, o)$ on $D\setminus \{o\}$. 
\end{proposition}
\begin{proof} By conditions \eqref{limgv}, for any $a\in \RR^+_*$, we have
\begin{equation*}
v(x)-ag_D(x,o)\leq O(1), \; o\neq x\to o;\quad
\limsup_{D\ni x\to \partial D} \bigl(v(x)-ag_D(x,o)\bigr)\leq 0. 
\end{equation*}
Hence the function $v-ag_D\in \sbh(D\setminus \{o\})$ has the removable singularity at the point $o$ \cite[Theorem 5.16]{HK}, and the function 
\begin{equation}\label{vo}
\begin{cases}
v-ag_D &\quad\text{on $D\setminus \{o\}$},\\ 
\limsup\limits_{o\neq x\to o}\big(v(x)-ag_D(x,o)\big) 
&\quad\text{at the point $o$}
\end{cases}
\end{equation}
is subharmonic on $D$ and $\limsup_{D\ni x\to \partial D}\big(v(x)-ag_D(x,o)\big)\leq 0$. By the maximum principle, the function \eqref{vo} is negative on $D$. Therefore, $v\leq ag_D(\cdot, o)$
on $D\setminus \{o\}$ \textit{for aarbitrary\/} $a>0$. 
Thus, $v\leq 0$ on $D$. In particular, for $v:=V-cg_D(\cdot, o)$, under the conditions \eqref{limgV}, we have \eqref{limgv}
and obtain $V-cg_D(\cdot, o)\leq 0$ on $D$.
\end{proof}
\begin{proposition}[{\rm on embedding}]\label{l2} 
Let $D$ be a domain with non-polar boundary, $S_o$ be a subset from\/ \eqref{S0seto} and $r, b_{\pm}$ are constants from\/  \eqref{bbpmr}. 
For any domain\/ $U_o$ satisfying 
\begin{equation}\label{o0SU}
o\in \Int S_o \Subset 
S_o^{\cup (3r)} \Subset U_o\Subset  D ,
\end{equation}
we can find a constant $B=\const^+_{o,S_o,r,U_o}\in \RR_*^+$ 
such that 
\begin{subequations}\label{VvU0}
\begin{align}
\mathcal V_{AS}^{\infty}(D\setminus U_o)\bigm|_{D\setminus S_o}&\overset{\eqref{PVlo},\eqref{{tf}bpm}}{\subset} \mathcal V_B:=
\sbh_{00}(D\setminus S_o ;r,-B< B)\cap C^{\infty} (D\setminus S_o ),
\tag{\ref{VvU0}$_{AS}$}\label{{VvU0}AS}
\\
\mathcal V_{J}^{\infty}(D\setminus U_o)\bigm|_{D\setminus S_o}& \overset{\eqref{PVloU}, \eqref{{tf}b+}}{\subset} \mathcal V_B^+:=
\sbh_{00}^+(D\setminus S_o;\leq B)\cap C^{\infty} (D\setminus S_o ),
\tag{\ref{VvU0}$_J$}\label{{VvU0}J}
\end{align}
\end{subequations}	
where, in the case\/ \eqref{{VvU0}AS}, 
it is assumed that the subset $S_o$ is connected.
\end{proposition}
\begin{proof} Let $V\in \mathcal V_{AS}^{\infty}(D\setminus U_o)\supset \mathcal V_{J}^{\infty}(D\setminus U_o)$.
\begin{lemma}\label{lemab}  If $V\in PAS_o(D)$, then 
$V\leq g_D(\cdot, o)$ on $D$.
\end{lemma}
\begin{proof}[of Lemma\/ {\rm \ref{lemab}}]
By the definition \eqref{PVl}  of Arens\,--\,Singer potentials, we have the conditions \eqref{limgV} of Proposition \ref{FLp}
with $c:=1$. Hence $V\leq g_D(\cdot, o)$ on $D$. 
\end{proof}

By Lemma \ref{lemab} we have
\begin{equation}\label{Babove}
\sup_{x\in S_o^{\cup(3r)}\setminus S_o}V(x)
\leq \sup_{x\in S_o^{\cup(3r)}\setminus S_o} g_D(x, o)=:B'=\const^+_{o,S_o,r}\in \RR_*^+.
\end{equation} 
If $V\in \mathcal V_{J}^{\infty}(D\setminus U_o)$, then $V\geq 0$ on $\RR_{\infty}^d\setminus \{0\}$, and we obtain 
\eqref{{VvU0}J} with $B:=B'=\const^+_{o,S_o, r}$.
In the case $V\in \mathcal V_{AS}^{\infty}(D\setminus U_0)
\setminus \mathcal V_{J}^{\infty}(D\setminus U_o)$, 
we use 
\begin{lemma}\label{lembl} Under the conditions \eqref{o0SU},
 there is a constant $B''\in \RR$ such that
\begin{equation}\label{pmulV}
\begin{split}
\inf_{x\in S_o^{\cup(3r)}\setminus \{o\}}V(x) &\geq B''=\const_{o,S_o,r,U_o}>-\infty
\\
\text{for every }V\overset{\eqref{PVl}}{\in}& PAS_o^1(D)\bigcap \har(U_o)\supset \mathcal V_{AS}^{\infty}(D\setminus U_0).
\end{split}
\end{equation} 
\end{lemma}
\begin{proof}[of Lemma\/ {\rm \ref{lembl}}] By Duality Theorem A in version \eqref{interAi0}--\eqref{PVl0},  its Riesz measure  $\varDelta_V=c_d \bigtriangleup V$ is a Arens\,--\,Singer probability measure. In particular,    
$\varDelta_V\in \Meas_{\comp}^{1+}(D\setminus U_o)$
and $V=\pt_{\varDelta_V-\delta_o}$. 
By Proposition \ref{pt_below} with $\mu\overset{\eqref{muLo}}{:=}\varDelta_V$ and  $L\overset{\eqref{muLo}}{:=}S_o^{\cup(3r)}$, 
we have
\begin{multline*}
\inf_{x\in S_o^{\cup(3r)}\setminus \{o\}}V(x)= \inf_{x\in S_o^{\cup(3r)}\setminus\{ o\}} {\pt}_{\varDelta_V-\delta_o}(x)
\\
\overset{\eqref{{pmul}o}}{\geq} \varDelta_V(\RR^d)k_{d-2}\bigl(\dist(S_o^{\cup(3r)},\supp \varDelta_V)\bigr)-
k_{d-2}\left(\sup_{x\in S_o^{\cup(3r)}}|x-o|\right)
\\
\geq k_{d-2}\bigl(\dist(S_o^{\cup(3r)}\setminus \{o\},\partial U_o)\bigr)-k_{d-2}\left(\sup_{x\in S_o^{\cup(3r)}}|x|+|o|\right)\\
\overset{\eqref{o0SU}}{=}
\const_{o,S_o,r,U_o}=B''\overset{\eqref{o0SU}}{>}-\infty.
\end{multline*}
\end{proof}

If we set $B\overset{\eqref{Babove}}{:=} \max \{B', (B'')^-\}$, then \eqref{Babove} and  \eqref{pmulV} give \eqref{{VvU0}AS}.

\end{proof}

\begin{proof}[of implication\/ {\rm {\bf s}\ref{{s}3}$\Rightarrow${\bf s}\ref{{s}1}}] 
According to \eqref{S0seto}, we can choose a point $o\in \Int S_0$ and a domain $U_o$ so that the relationships \eqref{o0SU}
are fulfilled, and $u(o)\neq -\infty$. The latter means
\begin{equation*}
-\infty <\int_{S_o} k_{d-2}\bigl(|x-o|\bigr) \dd \upsilon_u  
\Longrightarrow \int_{S_o} g_D(x,o) \dd \upsilon_u(x)<+\infty.  
\end{equation*}
Thus, by Lemma \ref{lemab}, we obtain
\begin{equation}\label{Vu}
\int_{S_o} V \dd \upsilon_u\leq \int_{S_o} g_D(\cdot, o)
\dd \upsilon_u=:C_1=\const^+_{o,S_o,u}\in \RR^+    
\quad\text{for all $V\in PAS_o(D)$}.
\end{equation}
Besides, by Lemma \ref{lembl}, we have 
\begin{equation}\label{VM}
-\infty <\const_{o,S_o,r, U_o,M}=B''\mu_M(S_o)\leq \int_{S_o} V\dd \mu_M
\quad\text{for all $V\in PAS^1_o(D)\bigcap \har (U_o)$}.  
\end{equation}
The condition {[\bf {s}\ref{{s}3}]} or the  
condition {[\bf {h}\ref{{h}3}]} means that
there is a constant $C_2\in \RR$ such that  
\begin{equation}\label{V2}
\begin{split}
\int_{D\setminus S_o} v \dd \upsilon_u
&\leq \int_{D\setminus S_o}v\dd \mu_M +C_2\\
\text{for all }v&\overset{\eqref{VvU0}}{\in}
\left[\begin{array}{cc}
\mathcal V_b^+ &\text{with $b\overset{\eqref{{VvU0}J}}{:=}b_+$ },  \\ 
 \mathcal V_b &\text{with $b\overset{\eqref{{VvU0}AS}}{:=}\min \{b_+,-b_-$\}},
\end{array} 
\right.\text{ \rm resp.}
\end{split}
\end{equation}
Let $B\in \RR_*^+$  be a constant from 
Proposition \ref{l2}  on embedding with inclusions 
\eqref{VvU0}. We multiply both sides of the inequality from \eqref{V2} by the number $B/b$ and obtain
\begin{equation*}
\int_{D\setminus S_o} v \dd \upsilon_u
\leq \int_{D\setminus S_o}v\dd \mu_M +BC_2
\text{ for all }v\overset{\eqref{VvU0}}{\in}
\left[\begin{array}{cc}
\mathcal V_B^+ &\text{in the case {[\bf {s}\ref{{s}3}]}},\\   
 \mathcal V_B &\text{in the case {[\bf {h}\ref{{h}3}]}},
\end{array} 
\right.\text{ \rm resp.}
\end{equation*}
Hence, by inclusions  \eqref{VvU0} from Proposition \ref{l2}  on embedding, we have 
\begin{equation*}\label{V2Bin}
\int_{D\setminus S_o} V \dd \upsilon_u
\leq \int_{D\setminus S_o}V\dd \mu_M +BC_2
\text{ for all }V\overset{\eqref{VvU0}}{\in}
\left[\begin{array}{cc}
\mathcal V_{J}^{\infty}(D\setminus U_o)&\text{in the case {[\bf {s}\ref{{s}3}]}},\\   
 \mathcal V_{AS}^{\infty}(D\setminus U_o) &\text{in the case {[\bf {h}\ref{{h}3}]}},
\end{array} 
\right.\text{ \rm resp.}
\end{equation*}
The last one after addition with two more inequalities  \eqref{Vu}--\eqref{VM} gives the inequality
\begin{subequations}\label{V2Bin+}
\begin{align}
\int_{D\setminus \{o\}} V \dd \upsilon_u
&\leq \int_{D\setminus \{o\}}V\dd \mu_M +C, 
\tag{\ref{V2Bin+}M}\label{{V2Bin+}M}
\\
\intertext{where  $C:=C_1+BC_2-B''\mu_M(S_o)=\const_{o,S_o,r,U_o,u,M}\in \RR$ independent of $V$,}
\text{for all }V&\overset{\eqref{VvU0}}{\in}
\left[\begin{array}{cc}
\mathcal V_{J}^{\infty}(D\setminus U_o)&\text{in the case {[\bf {s}\ref{{s}3}]}},\\   
 \mathcal V_{AS}^{\infty}(D\setminus U_o) &\text{in the case {[\bf {h}\ref{{h}3}]}},
\end{array} 
\right.\text{ \rm resp.}
\tag{\ref{V2Bin+}V}\label{{V2Bin+}V}
\end{align}
\end{subequations}
By Definition \ref{def:ab}, \eqref{{V2Bin+}M}--\eqref{{V2Bin+}V}
means that the measure $\mu_M$ is an affine balayage of the measure $\upsilon_u$ outside $\{o\}$ for the class
$\mathcal V_J^{\infty}(D\setminus U_o)$ from \eqref{PVloU}
in the case {[\bf {s}\ref{{s}3}]}
or for the class $\mathcal V_{AS}^{\infty}(D\setminus U_o)$ from \eqref{PVlo} in the case {[\bf {h}\ref{{h}3}]}, respectively. By Corollary \ref{JAS}, there is a function $h\in \sbh_*(D)$ in the case {[\bf {s}\ref{{s}3}]} 
or a function $h\in \har(D)$  in the case {[\bf {h}\ref{{h}3}]}, respectively, 
such that $u+h\leq M$ on $D$.  Thus, the implications 
{\rm {\bf s}\ref{{s}3}$\Rightarrow${\bf s}\ref{{s}1}}
and\/ {\bf h}{\rm \ref{{h}3}}$\Rightarrow${\bf h}{\rm\ref{{h}1}} are proved. 
\end{proof}

\section{Applications to the distribution of zeros
of holomorphic functions}\label{secH}
\setcounter{equation}{0}

\subsection{Additional definitions, notations and conventions}
For $n \in \NN$ we  denote by $\mathbb C^n$ the {\it $n$-dimensional complex  space over $\CC$\/} with the standard {\it norm\/} $|z|:=\sqrt{|z_1|^2+\dots+|z_n|^2}$ for $z=(z_1,\dots ,z_n)\in \CC^n$ and the distance function $\dist (\cdot, \cdot)$. By $\CC^n_{\infty}:=\CC^n\cup \{\infty\}$, and $\CC_{\infty}:=\CC_{\infty}^1$ we denote the \textit{one-point Alexandroff compactifications of $\CC^n$, and $\CC$;\/} $|\infty|:=+\infty$. 
If necessary, we identify $\CC^n$ and $\CC^n_{\infty}$
with $\RR^{2n}$ and $\RR^{2n}_{\infty}$ respectively (over $\RR$). In such case, the preceding terminology and concepts are transferred from $\RR^{2n}_{\infty}$  to $\CC^n_{\infty}$ naturally. 
For a proper subset $S\subset \CC^n_{\infty}$, the class $\Hol (S)$ consists of restrictions to $S$ of functions holomorphic in some (in general, its own for each function) open set $O\subset \CC_{\infty}^n$ containing $S$;
Throughout this Sec.~\ref{secH}, $D\neq \varnothing$ is a \textit{domain in\/} $\CC^n$.

\paragraph{\bf Zeros of holomorphic functions of several variables} 
{\cite[Ch.~1, 1,2]{Chirka}, \cite[\S~11]{Ronkin},\cite[Ch.~4]{Kha12}. The \textit{counting function\/ {\rm (or multiplicity function, or divisor)} of zeros} of function $f\in  \Hol_*(D)$ is a function
$\Zero_f\colon D\to \NN_0$ that can be defined as
\cite[1.5, Proposition 2]{Chirka} 
\begin{equation*}
\Zero_f(a):=\max\Bigl\{p\in \NN_0  \colon 
\limsup_{z\to a}\frac{|f(z)|}{|z-a|^p}<+\infty\Bigr\},
\quad a\in D.
\tag{\ref{nZn}Z}\label{Zerof0}
\end{equation*} 
with the \textit{support set}
$\supp \Zero_f=\bigl\{z\in D\colon f(z)=0\bigr\}$.
For $f=0\in \Hol(D)$, by definition, $\Zero_0\equiv +\infty$ on $D$. To each counting function of zeros $\Zero_f$ we assign the \textit{ counting measure of zeros} $n_{\Zero_f} \in \Meas^+ (D)$ defined as a Radon measure by the formulas 
\begin{subequations}\label{nZn}
\begin{align}
\begin{split}
n_{\Zero_f}(c)&:= \int_D c \dd n_{\Zero_f} 
:= \int_D c  \Zero_f \dd \varkappa_{2n-2}\\
\text{\it over all finite}&\text{ \it functions }c \in C_0(D):=
\bigl\{c\in C(D)\colon \supp c \Subset D\bigr\},
\end{split}
\tag{\ref{nZn}R}\label{{nZn}R}
\\
\intertext{or, equivalently, as a Borel measure on $D$ according to the rule} n_{\Zero_f} (S) &= \int_S
\Zero_f \dd \varkappa_{2n-2} 
\quad\text{\it for all open or closed subsets $S \Subset D$}.
\tag{\ref{nZn}B}\label{{nZn}B}
\end{align}
\end{subequations} 

\begin{PLf}[\cite{Lelong}]
Let $f \in \Hol_*(D)$. The following relations hold for the Riesz measure $\varDelta_{\ln |f|}$ of the function $\ln |f| 
\in \sbh_* (D)$:
\begin{equation}\label{nufZ}
\varDelta_{\ln |f|}\overset{\eqref{df:cm}}{:=}  c_{2n}\bigtriangleup \ln |f| \overset{\eqref{df:cm}}{=}\frac{(n-1)!}{2\pi^n \max\{1,2n-2\}}\bigtriangleup 
 \ln |f| {=}n_{\Zero_f}.
\end{equation}
\end{PLf}
Let $Z \colon D \to \RR^+$ be a function on $D$.
We call this function $Z$   a \textit{subdivisor of zeros\/} for function $f \in  \Hol(D)$ if $Z\leq  \Zero_f$ on $D$. The integrals with respect to a positive measure whose integrands contain a subdivisor are everywhere below treated as upper integrals \cite{Bourbaki}.
  
\paragraph{\bf Zeros of holomorphic functions of one variable}
\cite[0.1]{Kha12}. Let $D\subset \CC$,  $f\in  \Hol_*(D)$. Then the counting function\/ {\rm (or multiplicity function, or divisor)} of zeros of $f$ is the function
\begin{subequations}\label{Zf}
\begin{align}
\Zero_f(z)&\overset{\eqref{Zerof0}}{=}\max \Bigl\{ p\in \NN_0\colon 
\frac{f}{(\cdot-z)^p}\in \Hol(D)\Bigr\}, \quad z\in D,
\tag{\ref{Zf}Z}\label{{Zf}Z}
\\
\intertext{the counting measure of zeros defined as a Radon measure by the formulas}
\begin{split}
n_{\Zero_f}(c)&\overset{\eqref{{nZn}R}}{=}\sum_{z\in D} \Zero_f(z)c(z)=\int_D c\Zero_f\dd \varkappa_0\\
\text{\it for all}&\text{ finite functions }c \in C_0(D)
\text{ with }\supp \Zero_f:=\{z\in D\colon f(z)=0\},
\end{split}
\tag{\ref{Zf}R}\label{{Zf}R}
\\
\intertext{or as a Borel measure on $D$ according to the rule}
n_{\Zero_f}(S)&\overset{\eqref{{nZn}B}}{=}\sum_{z\in S} \Zero_f(z)\quad
\text{\it for all sets }S \Subset D.
\tag{\ref{Zf}B}\label{{Zf}B}
\end{align}
\end{subequations}
In this case, the support set $\supp \Zero_f$ is a locally finite set of isolated points in $D$.

An \textit{ indexed set} ${\sf Z}:=\{{\sf z}_{k}\}_{k=1,2,\dots}$ of points ${\sf z}_{k}\in D$
is {\it locally finite in} $D$ if \begin{equation*}
\# \{k \colon {\sf z}_{k}\in S\}<+\infty \quad \text{\it for each subset $S\Subset D$.}    
\end{equation*}
The \textit{counting measure\/}  $n_{\sf Z}\in \Meas^+(D)$ of this indexed set ${\sf Z}$ is defined as
\begin{equation}\label{Z}
n_{\sf Z}:=\sum_k \delta_{{\sf z}_k}, 
\text{ or, equivalently, }
n_{\sf Z}(S):=\sum_{{\sf z}_k\in S}1 \text { for any $S\subset D$}
\end{equation}
 Let ${\sf Z}$ and ${\sf Z}'$ be a pair of  indexed locally finite sets in $D$. By definition, ${\sf Z}={\sf Z}'$ if $n_{\sf Z}\overset{\eqref{Z}}{=}n_{{\sf Z}'}$, and
${\sf Z}'\subset {\sf Z}$ if $n_{{\sf Z}'}\overset{\eqref{Z}}{\leq} n_{\sf Z}$. 
An indexed set ${\sf Z}$ is the {\it zero set} of $f\in \Hol_*(D)$ if $n_{\sf Z}=n_{\Zero_f}$. 
A function $f\in \Hol(D)$ \textit{vanishes on\/}  a indexed set ${\sf Z}$ if ${\sf Z}\subset \Zero_f$. 

\subsection{Zero sets of holomorphic functions with restriction on their growth}\label{SSHn}

The following Theorem \ref{ThHol} develops results from \cite[Main Theorem, Theorems 1--3]{KhaRoz18}, and from \cite[Theorem 1]{MenKha19}. Both the integrals on the right-hand sides of inequalities \eqref{in:HZ} and \eqref{in:HZb}, and the pair of integrals in inequality \eqref{in:HZb+} below  are, generally speaking, upper integrals in the sense N.~Bourbaki \cite{Bourbaki}.

\begin{theorem}\label{ThHol} Let $D\neq \varnothing$ be a domain in $\CC^n$, 
\begin{equation}\label{M}
M_+\in \sbh_*(D), 
\quad  M_-\in \sbh_*(D), \quad M:=M_+-M_-\in \dsbh(D)
\end{equation}
are functions   with the Riesz measures $\mu_{M_+},\mu_{M_-}\in \Meas^+(D)$, and the Riesz  charge $\mu_M=\mu_{M_+}-\mu_{M_-}\in \Meas (D)$, respectively, and $f\in \Hol_* (D)$ be a function such that
\begin{equation}\label{fMD}
|f|\leq \exp M \quad \text{on $D$}.
\end{equation}
Then
\begin{enumerate}[{\rm [{\sf Z}I]}]
\item\label{ZI} 
For any  connected set $S_o\Subset D$ with $\Int S_o\neq \varnothing$  from\/ \eqref{S0seto} and  for any numbers $r, b_{\pm}$ from\/ \eqref{bbpmr+}, i.\,e.,
$0<3r<\dist (S_o, \partial D)$, $-\infty<b_-<0<b_+<+\infty$, 
there is a constant\/ 
\begin{equation}\label{C1}
C=\const_{u,M,S_o, r, b_{\pm}}\in\RR
\end{equation} 
such that
\begin{equation}\label{in:HZ}
\int_{D\setminus S_o} v \Zero_f \dd \varkappa_{2n-2}
\leq \int_{D\setminus S_o^{\cup(3r)}} v \dd \mu_M
+ \int_{S_o^{\cup(3r)}\setminus S_o} (-v) \dd \mu_{M_-}+C  
\end{equation}
for all functions $v\in \sbh_{+0}^{\uparrow}(D\setminus S_o;\circ r, b_-< b_+)$.
\item\label{ZII} For any  connected set $S_o$  from \eqref{S0seto} and  numbers $b_{\pm},r$ from\/ \eqref{bbpmr+}, i.\,e.,
$0<3r<\dist (S_o, \partial D)$, $-\infty<b_-<0<b_+<+\infty$, 
there is a constant\/ \eqref{C1} such that
\begin{equation}\label{in:HZb}
\hspace{-7pt}\int_{D\setminus S_o} v \Zero_f \dd \varkappa_{2n-2}
\leq \int_{D\setminus S_o} v \dd \mu_M+C 
\text{ for all $v\in \sbh_{+0}^{\uparrow}(D\setminus  S_o;  r, b_-< b_+)$.}
\end{equation}
 
\item\label{ZIII} For any  set $S_o$  from \eqref{S0seto}, constant $b\in \RR_*^+$, and subdivisor 
${\sf Z}\leq \Zero_f$,  there is a constant  
$C=\const_{u,M,S_o}\in\RR$ such that
\begin{equation}\label{in:HZb+}
\int_{D\setminus S_o} v {\sf Z} \dd \varkappa_{2n-2}
\leq \int_{D\setminus S_o} v \dd \mu_M+C 
\quad \text{for all $v\in \sbh_0^{+\uparrow}(D\setminus S_o; \leq b)$.}
\end{equation}
\end{enumerate}
Besides, the implication   {\sf [Z\ref{ZI}]}$\Rightarrow${\sf [Z\ref{ZII}]} is true.

\end{theorem}
\begin{proof} We can rewrite \eqref{fMD} as
\begin{equation}
\sbh_*(D)\ni u:=\ln |f|+M_-\leq M_+\overset{\eqref{M}}{\in} \sbh_*(D). 
\end{equation}
By Poincar\'e\,--\,Lelong formula \eqref{nufZ}, and by implication {\bf h\ref{{h}1}}$\Rightarrow${\bf h\ref{{h}2+}} of  Criterium \ref{crit2} together  with Remark \ref{rembD}, there is a constant \eqref{C1} such that 
\begin{multline*}
\int_{D\setminus S_o} v \Zero_f \dd \varkappa_{2n-2}
+ \left(\int_{S_o^{\cup(3r)}\setminus S_o}+\int_{D\setminus S_o^{\cup(3r)}}\right) \, v \dd \mu_{M_-}\\
\overset{\eqref{nufZ}}{=}\int_{D\setminus S_o} v \dd (\varDelta_{\ln|f|}+\mu_{M_-})
\overset{\eqref{almB}}{\leq} \int_{D\setminus S_o^{\cup(3r)}} v \dd \mu_M+C.
\end{multline*}
for all functions $v\in \sbh_{+0}^{\uparrow}(D\setminus S_o;\circ r, b_-< b_+)$. Hence we have the statement {\rm [{\sf Z}\ref{ZI}]}, \eqref{in:HZ}. 

Similarly, by Definition \ref{def:ab} of affine balayage, by  Poincar\'e\,--\,Lelong formula \eqref{nufZ}, and by implication {\bf h\ref{{h}1}}$\Rightarrow${\bf h\ref{{h}2}} 
of  Criterium \ref{crit2} together with Remark \ref{rembD}, we obtain the statement {\rm [{\sf Z}\ref{ZII}]}, \eqref{in:HZb}. Besides, the implication {\sf [Z\ref{ZI}]}$\Rightarrow${\sf [Z\ref{ZII}]} follows from  the estimate 
\begin{equation*}
\int_{S_o^{\cup(3r)}\setminus S_o}|v|\dd |\mu_M|\overset{\eqref{{tf}bpm0}}{\leq}
\max \{b_+,-b_-\} |\mu_M|(S_o^{\cup(3r)}\setminus S_o)
\end{equation*}
true for each test function $v\overset{\eqref{{tf}bpm0}}{\in}\sbh_{+0}(D\setminus S_o; r,b_-< b_+)$.

Finally, by Definition \ref{def:ab} of affine balayage, by  Poincar\'e\,--\,Lelong formula \eqref{nufZ}, and by implication {\bf s\ref{{s}1}}$\Rightarrow${\bf s\ref{{s}2}}
of  Criterium \ref{crit1} together with Remark \ref{rembD}, we get 
the statement {\rm [{\sf Z}\ref{ZIII}]}, \eqref{in:HZb+},
 since
\begin{equation*}
\int_{D\setminus S_o} v {\sf Z} \dd \varkappa_{2n-2}\leq \int_{D\setminus S_o} v \Zero_f \dd \varkappa_{2n-2} 
\end{equation*} 
for every\textit{ positive\/} function $v\in \sbh_0^{+\uparrow}(D\setminus S_o; \leq b)$.
\end{proof}
\begin{remark} For $n>1$, if the majorizing function $M$ 
in \eqref{fMD} is plurisubharmonic, then the scale of necessary conditions for the distribution of zeros of the holomorphic function $f$ can be much wider than that presented in Theorem \ref{ThHol}. It should include other characteristics related to the Hausdorff measure of smaller dimension than $2n-2$. We plan to consider this case elsewhere. In particular, analytical and polynomial disks should play a key role in this case (see \cite[Ch.~3]{Krantz}, \cite{Schachermayer}, 
\cite{Poletsky93}, \cite{Po99},  \cite[\S~4]{KhaKha19}, etc.). 
\end{remark}

\subsection{The case of a finitely connected domain $D\subset \CC$}\label{Dfd}

 \paragraph{\bf Agreement.} Throughout this Subsec. \ref{Dfd}, the domain $D\subset \CC$ is\textit{ finitely connected in\/} $\CC_{\infty}$ with number of component $\#\Conn_{\CC_{\infty}} 
 \partial D<+\infty$,  among which \textit{there is at least one component containing two different points.\/} Then the  \textit{boundary} $\partial D$ of such  domain $D$ is a \textit{non-polar set.}

\begin{lemma}[{\rm \cite[Lemma 2.1]{Kha07}}]\label{lem:har} If $h\in \har (D)$, then there are a real  number $c <\#\Conn_{\CC_{\infty}} \partial D-1$     
and a function $g\in \Hol (D)$ without zeros in $D$, i.\,e.
with $n_{\Zero_g}=0$, such that
\begin{equation}\label{estHar}
\ln \bigl|g(z)\bigr|\leq h(z)+c^+\ln\bigl(1+|z|\bigr)\quad\text{for all $z\in D$.} 
\end{equation}
If\/ $\clos D\neq \CC_{\infty}$, then on the right-hand side of\/ \eqref{estHar} we can put\/ $c:=0$. 
\end{lemma}

\begin{theorem}\label{th:3}  Let $M\in \dsbh(D)$ be a function from\/ \eqref{M} with $M_+\in C(D)$, and\/  ${\sf Z}:=\{{\sf z}_{k}\}_{k}$ be an  indexed 
locally finite  set in $D$ of points ${\sf z}_{k}\in D$.
If there are a connected set\/ $S_o$ as in \eqref{S0seto},  numbers $b_{\pm}$, $r$ as in\/ \eqref{bbpmr+}, and a constant $C\in \RR$ such that    
\begin{equation}\label{in:HZb+Z}
\begin{split}
\sum_{{\sf z}_k\in D\setminus S_o} 
v ({\sf z}_k)&\leq \int_{D\setminus S_o} v \dd \mu_M+C \\
\text{for all } v&\in \sbh_{00}(D\setminus S_o; r,b_-< b_+)\bigcap  C^{\infty}(D\setminus S_o),
\end{split}
\end{equation}
then there are a real number  $c<\#\Conn \partial D-1$ 
and a function  $f\in \Hol(D)$ with zero set 
${\sf Z}$ such that 
\begin{equation}\label{fM}
\ln \bigl|f(z)\bigr|\leq M(z)+c^+\ln\bigl(1+|z|\bigr) \quad \text{for all
$z\in D$}, 
\end{equation}
where we additionally set  $c:=0$ if $\clos D\neq \CC_{\infty}$.
\end{theorem}
\begin{proof} We can rewrite relation \eqref{in:HZb+Z} as
\begin{equation*}
\int_{D\setminus S_o}v \dd (n_{\sf Z}+\mu_{M_-})\leq \int_{D\setminus S_o} v \dd \mu_{M_+}+C 
\end{equation*}
where the constant $C$ is independent of $v\in \sbh_{00}(D\setminus S_o; r,b_-< b_+)\bigcap  C^{\infty}(D\setminus S_o)$. By Definition \ref{def:ab}, this means that
 $\mu_{M_+}$ is an affine  balayage of  $n_{\sf Z}+\mu_{M_-}\in \Meas^+(D)$ for the class $\sbh_{00}(D\setminus S_o; r,b_-< b_+)\bigcap  C^{\infty}(D\setminus S_o)$  outside $S_o$. There is a function  $u\in \sbh_*(D)$ with the  Riesz measure $n_{\sf Z}+\mu_{M_-}$ \cite[Theorem 1]{Arsove}. It follows from  implication {\bf h\ref{{h}3}}$\Rightarrow${\bf h\ref{{h}1}} of Criterium \ref{crit2}  that there exists a function $h\in \har(D)$ such that
$u+h\leq M_+$. According to one of Weierstrass theorems, there is  a function $f_{\sf Z}\in \Hol(D)$ with the zero set  ${\sf Z}$. Hence, using  Weyl's lemma for Laplace equation,  we have a representation of the function $u=\ln |f_{\sf Z}|+M_-+H$, where $H\in \har(D)$.
Therefore
\begin{equation}\label{fH}
 \ln |f_{\sf Z}|+H+h\leq M_+-M_-\overset{\eqref{M}}{=}M \quad\text{on $D$}. 
\end{equation}
By Lemma  \ref{lem:har}, there is a function $g\in \Hol(D)$ \textit{without zeros\/} such that 
\begin{equation}\label{gH}
\ln |g|\leq H+h+c^+\ln\bigl(1+|\cdot|\bigr)\quad \text{on $D$},
\end{equation} 
where $c$ is a constant from  Lemma \ref{lem:har}.  Addition \eqref{fH} and \eqref{gH} gives 
\begin{equation*}
\ln |f_{\sf Z}|+\ln |g|\leq M+c^+\ln\bigl(1+|\cdot|\bigr) \quad\text{on $D$.}
\end{equation*}
 If we set $f:=gf_{\sf Z}\in \Hol(D)$, then ${\sf Z}$ is the zero set of $f$, and  we have  \eqref{fM}. 
\end{proof}
The intersection of Theorem \ref{th:3} with Theorem \ref{ThHol}, [{\sf Z\ref{ZII}}],  gives the following
\begin{Criterium}\label{crit3} Let the conditions of Theorem \ref{th:3} be satisfied.
Besides, let the domain $D$ is simple connected or $\CC_{\infty}\setminus D\neq \varnothing$ under restrictions from our Agreement.   
Then the following  four assertion are equivalent:
\begin{enumerate}[{\bf [z1]}]
\item\label{z1} There is a function $f\in \Hol(D)$ with the zero set ${\sf Z}$ such that
$|f|\leq \exp M$ on $D$. 

\item\label{z2} For any  connected set $S_o$  from\/ \eqref{S0seto} and for any  numbers $b_{\pm},r$ from\/ \eqref{bbpmr+}, 
there is a constant\/ $C$ as in\/  \eqref{C1} such that
\begin{equation}\label{in:HZ1}
\sum_{{\sf z}_k\in D\setminus S_o} v ({\sf z}_k)
\leq \int_{D\setminus S_o^{\cup(3r)}} v \dd \mu_M
+ \int_{S_o^{\cup(3r)}\setminus S_o} (-v) \dd \mu_{M_-}+C  
\end{equation}
for all functions $v\in \sbh_{+0}^{\uparrow}(D\setminus S_o;\circ r, b_-< b_+)$.

\item\label{z3} For any  connected set $S_o$  from \eqref{S0seto} and for any  numbers $b_{\pm},r$ from\/ \eqref{bbpmr+}, 
there is a constant\/ $C$ as in \eqref{C1} such that
\begin{equation}\label{in:HZb1}
\sum_{{\sf z}_k\in D\setminus S_o} v ({\sf z}_k)
\leq \int_{D\setminus S_o} v \dd \mu_M+C
\end{equation}
for all $v\in \sbh_{+0}^{\uparrow}(D\setminus  S_o;  r, b_-< b_+)$.
\item\label{z4} There are  connected set $S_o$  as in \eqref{S0seto},  numbers $b_{\pm},r$ as in\/ \eqref{bbpmr+}, and a constant\/ $C$  such that we have 
\eqref{in:HZb1} for all functions 
$ v\in \sbh_{00}(D\setminus S_o; r,b_-< b_+)\bigcap  C^{\infty}(D\setminus S_o)$.
\end{enumerate}
\end{Criterium}
\begin{proof} The implications   {\bf [z\ref{z1}]}$\Rightarrow${\bf [z\ref{z2}]}
$\Rightarrow${\bf [z\ref{z3}]}$\Rightarrow${\bf [z\ref{z4}]}  follows from Theorem \ref{ThHol} with   {\sf [Z\ref{ZI}]}$\Rightarrow${\sf [Z\ref{ZII}]} and Proposition \ref{pr:12}. The implications   {\bf [z\ref{z4}]}$\Rightarrow${\bf [z\ref{z1}]}
follows from Theorem \ref{th:3}, where  $c=0$ in \eqref{fM} according to the properties of the domain $D$.
\end{proof}

\begin{remark} See Remark \ref{rem:7} on a 
special case of Criterium  \ref{crit3}  announced in  \cite[Theorem 2]{MenKha19}.
\end{remark}

\begin{remark}   The work \cite[Theorems 2, 4,5]{KhaKha19} contains a wide range of necessary or sufficient conditions under which there exists a function $f\in \Hol_*(D)$
 that \textit{vanishes on\/} ${\sf Z}$  and satisfies the inequality $|f|\leq \exp M$ on $D$.  
These results do not follow directly from our Criterium \ref{crit3}.
\end{remark}

\paragraph{\bf Acknowledgements} 
The work was supported by a Grant of the Russian Science Foundation (Project No. 18-11-00002, Bulat Khabibullin, sections 1--11), and by a Grant   of the Russian Foundation of Basic Research (Project No. 19-31-90007, Enzhe Menshikova, section 12).

{\sl Bashkir State University, Ufa, Bashkortostan, Russian Federation}
\end{document}